\documentclass{amsart}
\usepackage{amsmath,amssymb,bbm,mathdots,rotating}
\usepackage{dsfont}
\usepackage{graphicx, xcolor}
\usepackage{tikz}
\usetikzlibrary{calc}
\usepackage{tabularx}
\usepackage{subcaption}

\usepackage[backend=biber,style=alphabetic,]{biblatex}

\theoremstyle{definition}
\newtheorem{definition}{Definition}[section]
\newtheorem{theorem}[definition]{Theorem}
\newtheorem{lemma}[definition]{Lemma}
\newtheorem{example}[definition]{Example}
\newtheorem{proposition}[definition]{Proposition}
\newtheorem{corollary}[definition]{Corollary}
\newtheorem*{corollary*}{Corollary}

\newtheorem{principle}{Principle}
\newtheorem{conjecture}[definition]{Conjecture}

\begin{document}

\def\spanset{\operatorname {span}}
\def\col{\operatorname {col}}
\def\row{\operatorname {row}}
\def\design{\operatorname {design}}
\def\dist{\operatorname {dist}}
\def\eff{\operatorname {efficacy}}
%Greek
\def\a{\alpha}
\def\b{\beta}
\def\c{\gamma}
\def\d{\delta}
\def\e{\varepsilon}
\def\g{\gamma}
\def\l{\lambda}
\def\L{\Lambda}
\def\o{\omega}
\def\s{\sigma}
\def\t{\tau}
\renewcommand{\phi}{\varphi}
%emptyset
\def\vn{\varnothing}
%fat numbers
\def\ones{\mathbbm{1}}
\def\zeros{\mathbf{0} }
\def\cR{{\mathcal R}}
\def\RR{{\mathbb R}}
\def\SS{{\mathbb S}}
\def\CC{{\mathbb C}}
\def\cO{{\mathcal O}}
\def\cA{{\mathcal A}}
\colorlet{cyan}{black}
\title{Codes, Cubes, and Graphical Designs}
\keywords{Graphical Designs, Error Correcting Codes, Hamming Code, Hoffman Bound, Cheeger Bound, Graph Laplacian, Graph Sampling, Stable Sets, Association Schemes, t-Designs, Spherical Designs, Sobolev-Lebedev Quadrature.}
\subjclass[2010]{05C50, 05E30, 68R10} 
%%05C50 = Combo: graphs + linear algebra
%% 05E30=Association schemes, strongly regular graphs 
%% 68R10 = Discrete mathematics in relation to computer science: graph theory
\author[]{\vspace{-.1 in}Catherine Babecki}
\thanks{Research partially supported by the U.S. National Science Foundation grant 
DMS-1719538}
\address{Department of Mathematics, University of Washington, Seattle, WA 98105 USA}
\email{cbabecki@uw.edu}

\begin{abstract}
\vspace{-.1 in}
Graphical designs are an extension of spherical designs to functions on graphs. We connect linear codes to graphical designs on cube graphs, and show that the Hamming code in particular is a highly effective graphical design. We show that even in highly structured graphs, graphical designs are distinct from the related concepts of extremal designs, maximum stable sets in distance graphs, and $t$-designs on association schemes.
\end{abstract}
\maketitle

\section{Introduction to Graphical Designs}
In this paper, we  establish several new results about graphical designs, which are an analog of the well-known spherical designs to the discrete domain of graphs. We show that linear codes are effective graphical designs on the cube graph, which is the graph that has the vertices and edges of the unit cube. We then show that graphical designs are not the same as several related combinatorial concepts.  We begin this section with the history and motivation for defining graphical designs, and then introduce the graphical design problem. 

\begin{figure}[h!]
    \centering
    \includegraphics[scale =.2]{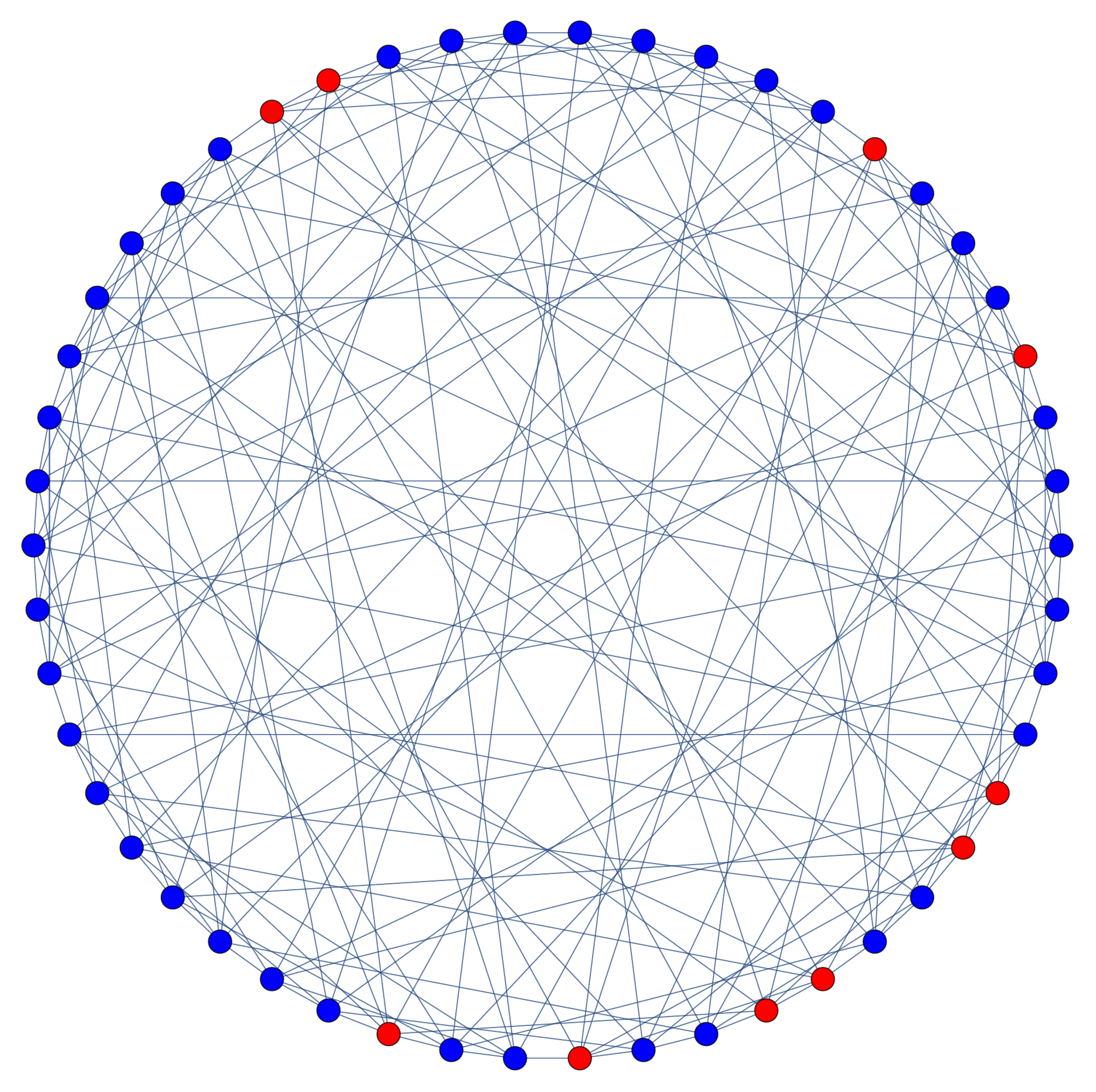}
    \caption{The Hoffman-Singleton graph on 50 vertices.  The 10 red vertices form a graphical design which integrate the first 49 eigenvectors of the graph Laplacian in a suitable basis. The induced subgraph of this design is 3-regular.}
    \label{bigboy}
\end{figure}

\subsection{Spherical Designs}

  We begin with a quick introduction to quadrature rules on the sphere. For more complete references on spherical harmonics, see, for instance, \cite{spherical} or \cite{HFT}.

 It may be impossible or computationally prohibitive to exactly integrate a function on the sphere $\SS^{d-1}$, so it is useful to find good numerical approximations. Spherical quadrature approaches this problem by looking for well distributed points over the sphere. A \emph{quadrature rule} is a set of points  $\{x_1,\ldots, x_N\} \subset \SS^{d-1}$  and weights $\a_i \in \RR$ chosen so that $$ \frac{1}{|\SS^{d-1}|} \int_{\SS^{d-1}} f(x) \ dx \approx \sum_{i=1}^N \a_i f(x_i)$$ whenever $f$ is smooth in some suitable way. For instance,  a \emph{spherical $t$-design} is a quadrature rule which integrates all polynomials up to degree $t$ exactly (\cite{DGSspherical}).  As one might guess, there are numerous ideas about how best to distribute the points. The strategy of Sobolev-Lebedev quadrature is to fix a finite-dimensional vector space of functions and find points which integrate the entire vector space exactly. This raises the question, what functions should one try to integrate? 
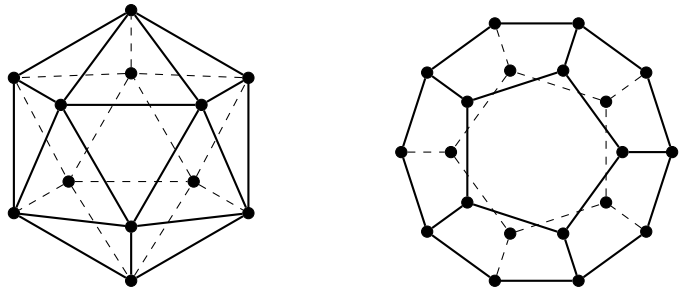
\begin{figure}[h!]
\begin{align*}
& \begin{tikzpicture}[scale=0.6]
  \tikzstyle{every node}=[circle,draw =black, fill=black ,inner sep=1.5pt]%%Icoso
    \foreach \y[count=\a] in {10,9,4}
      {\pgfmathtruncatemacro{\kn}{120*\a-90}
       \node at (\kn:3) (b\a){} ;}
    \foreach \y[count=\a] in {8,7,2}
      {\pgfmathtruncatemacro{\kn}{120*\a-90}
       \node at (\kn:1.8) (d\a) {};}
    \foreach \y[count=\a] in {1,5,6}
      {\pgfmathtruncatemacro{\jn}{120*\a-30}
       \node at (\jn:1.6) (a\a) {};}
    \foreach \y[count=\a] in {3,11,12}
      {\pgfmathtruncatemacro{\jn}{120*\a-30}
       \node at (\jn:3) (c\a) {};}
  \draw[dashed] (a1)--(a2)--(a3)--(a1);
  \draw[thick] (d1)--(d2)--(d3)--(d1);
  \foreach \a in {1,2,3}
   {\draw[dashed] (a\a)--(c\a);
   \draw[thick] (d\a)--(b\a);}
   \draw[thick] (c1)--(b1)--(c3)--(b3)--(c2)--(b2)--(c1);
   \draw[thick] (c1)--(d1)--(c3)--(d3)--(c2)--(d2)--(c1);
   \draw[dashed] (b1)--(a1)--(b2)--(a2)--(b3)--(a3)--(b1);
\end{tikzpicture}
&&   
\begin{tikzpicture}[scale=0.6] %%dodeca
\tikzstyle{every node}=[circle,draw =black, fill=black ,inner sep=1.5pt]
    \foreach \y in {1,2,3,4,5,6,7,8,9,10}
        {\node at (36*\y +72:3) (a\y) {};}
    \foreach \y in {11,12,13,14,15}
        {\node at (72*\y :1.9) (b\y) {};}
    \foreach \y in {16,17,18,19,20}
        {\node at (72*\y+ 36 :1.9) (c\y) {};}
        %%edges
    \draw[dashed] (c16)--(c17)--(c18)--(c19)--(c20)--(c16);
    \draw[thick] (b11)--(b12)--(b13)--(b14)--(b15)--(b11);
    \draw[thick] (a1)--(a2)--(a3)--(a4)--(a5)--(a6)--(a7)--(a8)--(a9)--(a10)--(a1);
    \draw[thick] (a10) -- (b11);
    \draw[thick] (a8) -- (b15);
    \draw[thick] (a6) -- (b14);
    \draw[thick] (a4) -- (b13);
    \draw[thick] (a2) -- (b12);
    \draw[dashed] (a1)--(c16);
    \draw[dashed] (a9)--(c20);
    \draw[dashed] (a7)--(c19); 
    \draw[dashed] (a5)--(c18); 
    \draw[dashed] (a3)--(c17); 
\end{tikzpicture}
\end{align*}
\caption{The vertex sets of the icosahedron and dodecahedron both form spherical 5-designs,  integrating a 36-dimensional vector space of functions exactly. The icosahedron is also a Sobolev-Lebedev quadrature rule.}
\end{figure}

Sobolev suggested that a suitable vector space of functions should respect the symmetries of the sphere (see \cite{sobolev, Lebedev}), leading us to spherical harmonics.  Spherical harmonics are the eigenfunctions of the Laplace-Beltrami operator $\Delta$, where a function $f:\SS^{d-1} \to \RR$ is called an eigenfunction of $\Delta$ if $\Delta f= \lambda f$ for some eigenvalue $\l \in \RR$.  The eigenfunctions of $\Delta$ form an orthogonal basis for $L^2(\SS^{d-1})$ consisting of harmonic, homogeneous polynomials.  In the one dimensional case, spherical harmonics lead to the standard trigonometric polynomials; expanding a function in this basis is its Fourier series.  Spherical harmonics can be naturally ordered by frequency: a spherical harmonic with eigenvalue $\l$ has a low frequency if $|\l|$ is small. The idea behind Sobolev-Lebedev quadrature is to integrate all low-frequency spherical harmonics up to a chosen threshold for $|\l|$. 

\subsection{Graphical Designs}

The graphical design problem is an extension of spherical quadrature to functions on graphs, first introduced by Steinerberger in \cite{graphdesigns}. This notion of graphical design is distinct from the graphical designs in design theory (see \cite[p. 222]{Stinson}).  Let $G = (V,E)$ be a graph with vertex set $V= [n] := \{1,\ldots, n\}$ and edge set $E$, and consider a class of functions $\phi_i: V\to \RR$ which are ``smooth" with respect to the geometry of $G$, a notion which we will make precise shortly. A quadrature rule on $G$ is a subset $W\subset V$ and a set of weights $\{\a_w\}_{w\in W}$ such that for each $\phi_i,$
$$\frac{1}{|V|}\sum_{v \in V} \phi_i(v) = \sum_{w \in W} \a_w\phi_i(w). $$
In this paper, we only consider equal weights; we would like that for each $\phi_i,$
$$\frac{1}{|V|}\sum_{v \in V} \phi_i(v) = \frac{1}{|W|}\sum_{w \in W} \phi_i(w). $$ 

\textcolor{cyan}{To mirror quadrature on the sphere, we use eigenvectors of the discrete Laplacian of $G=(V,E)$ as the class of smooth functions. We define the Laplacian as 
\begin{align*}
    &L = AD^{-1} - I, &&\text{which is equivalent to}  \quad  (Lf)(u) = \sum_{v: uv \in E} \left( \frac{f(v)}{\deg(v)}- \frac{f(u)}{\deg(u)}\right).
\end{align*}
The adjacency matrix $A$ is defined by $A_{ij} = 1$ if $ij\in E$ and 0 otherwise, $D$ is the diagonal matrix with $D_{ii} = \deg(v_i)$, and $I$ is the $n \times n$ identity matrix. In the smooth case, Taylor expansion shows that the Laplace-Beltrami operator is essentially the average value in a neighborhood of a point. This graph Laplacian analogously captures averaging over the neighborhood of a vertex. }

The definition of a graphical design is dependent on the choice of Laplacian. Other common graph Laplacians include $D-A$ and the normalized Laplacian $I - D^{-1/2}A D^{-1/2}$.  We refer the reader to \cite{ChungSpectral} for more on the various graph Laplacians. We focus on regular graphs where these notions of Laplacian coincide; if $G = (V,E)$ is $\d$-regular, then $$L\phi = \l \phi,  \iff (D-A) \phi = -\d\l \phi \iff (I-D^{-1/2}A D^{-1/2})\phi = -\l \phi.$$
 
 Note that the operator $AD^{-1}$ is a nonnegative Markov matrix, and hence has spectrum $\sigma(AD^{-1}) \subseteq [-1,1]$. Thus $\sigma(L) \subseteq [-2,0]$. We order the eigenvalues as in \cite{graphdesigns}, where $\l_1$ has the lowest frequency and $\l_n$ has the highest:  $$|\l_1+1| \geq |\l_2+1| \geq \ldots \geq |\l_n+1|.$$  We now provide some intuition for how this is analogous to the frequency ordering on spherical harmonics.  In a regular graph, $L \ones = 0 = \l_1$, where $\ones$ is the all-ones vector.  The eigenvector $\ones$ is constant across all vertices, and hence is the ``smoothest" possible with respect to the structure of any graph. By smooth, we mean the function does not change dramatically across vertices which are highly connected or in some sense representative of the same part of the graph. Likewise, if $G$ is bipartite, then $L$ has an eigenvector $\phi$ with eigenvalue $\l_2=-2$, and $\phi$ is constant on each part of the bipartition. Eigenvectors with eigenvalue near $-1$ do not exhibit this type of smoothness (see Figure 3).
 
\begin{figure}[h!]
\centering
\begin{tabular}{cc}
\subfloat[The 1$^\text{st}$  eigenvector.]{\includegraphics[scale = .18]{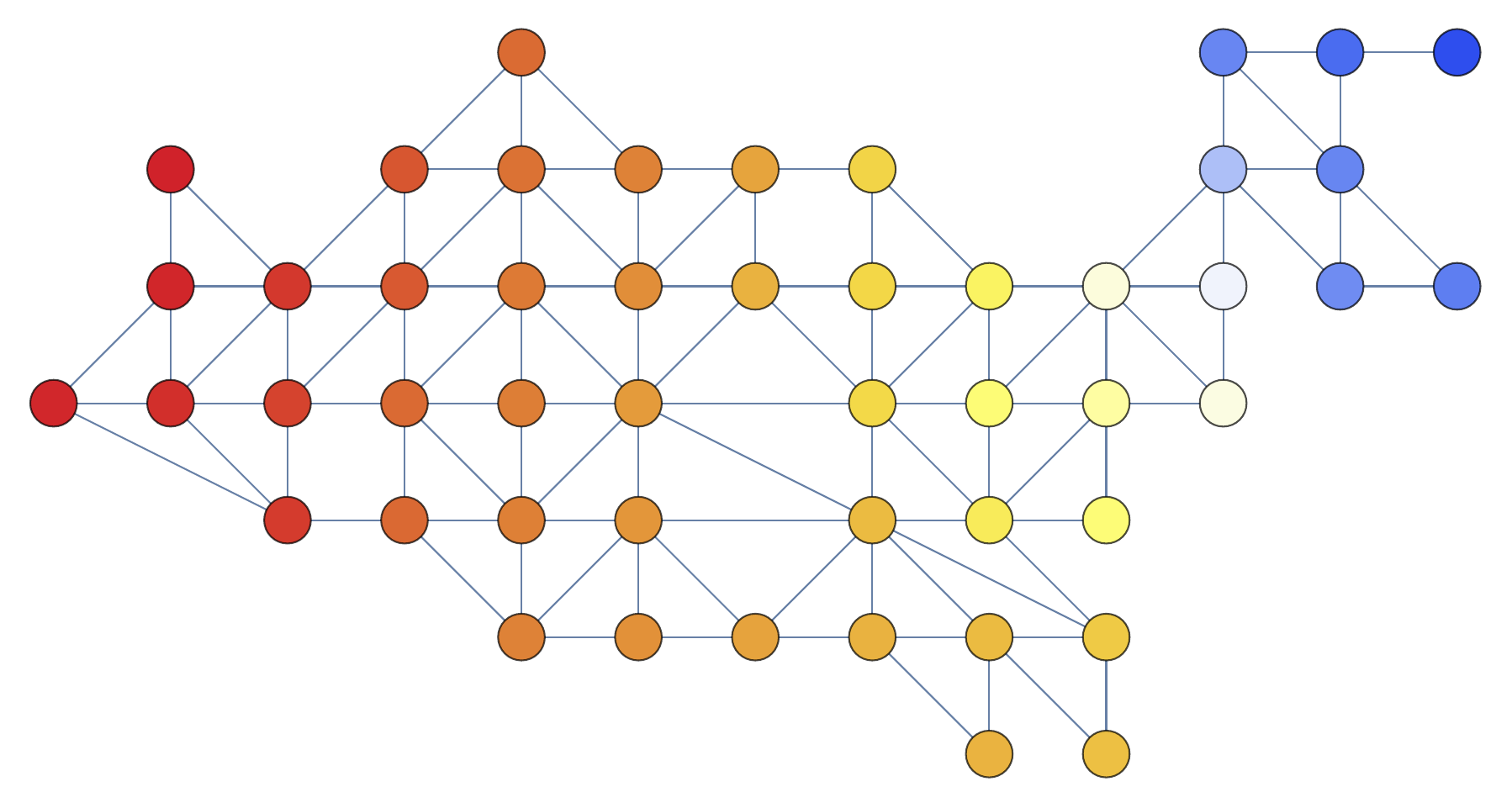}} & 
\subfloat[The 2$^\text{nd}$  eigenvector.]{\includegraphics[scale = .18]{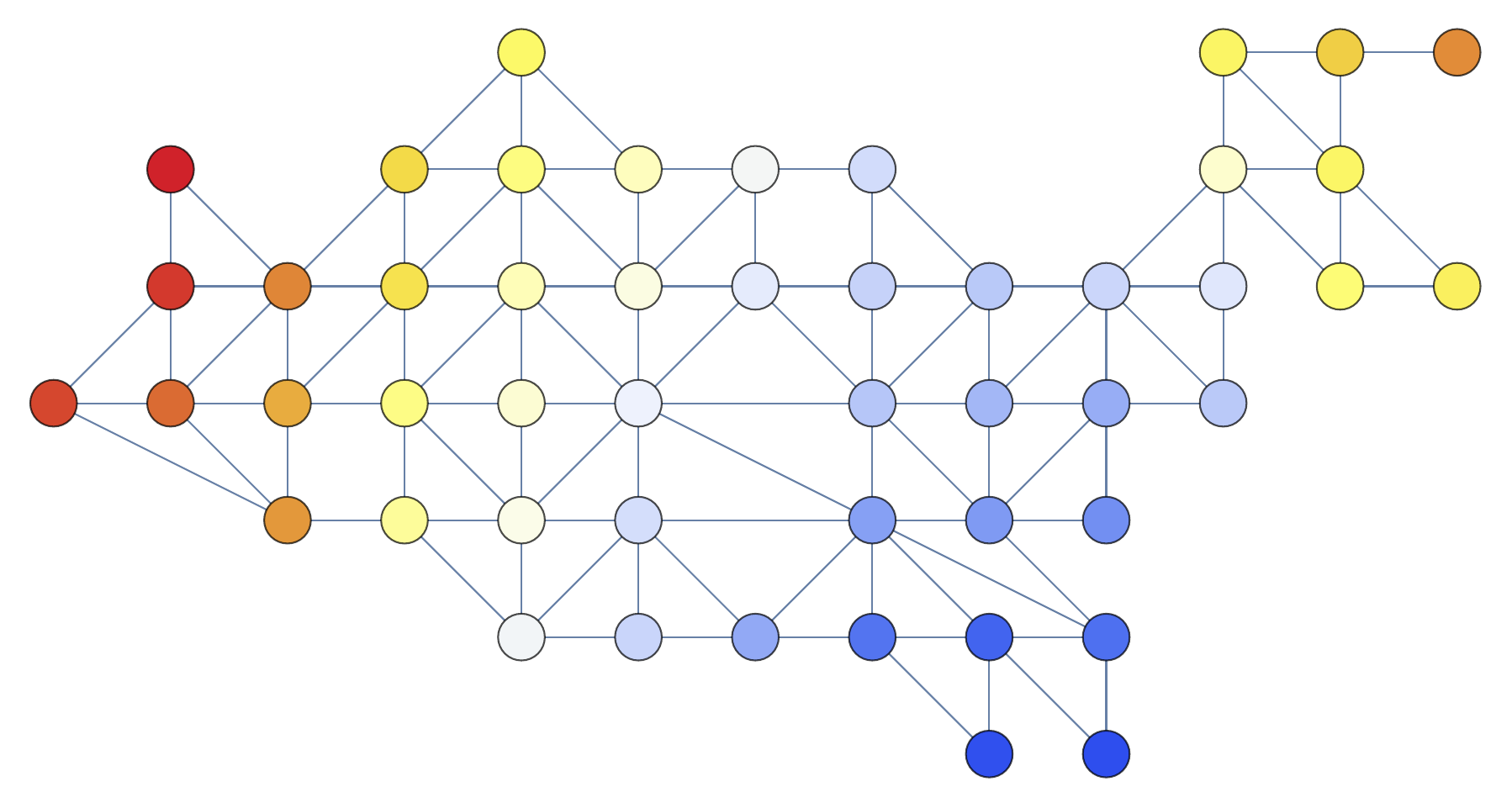}} \\
\subfloat[The 11$^{\text{th}}$ eigenvector.]{\includegraphics[scale = .18]{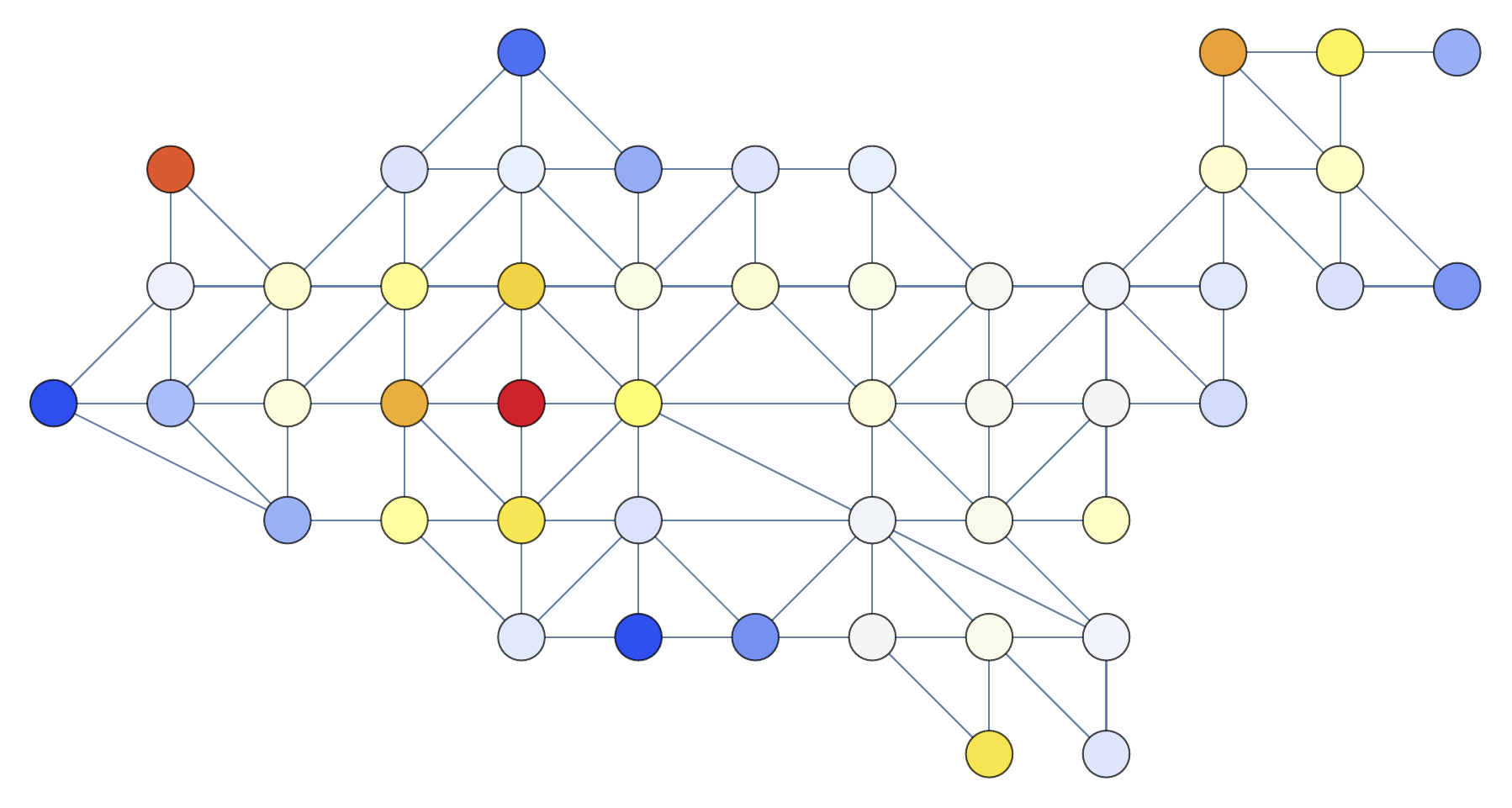}} & 
\subfloat[Average Annual Precipitation in the Contiguous US, 1981-2010]{\includegraphics[scale =.18]{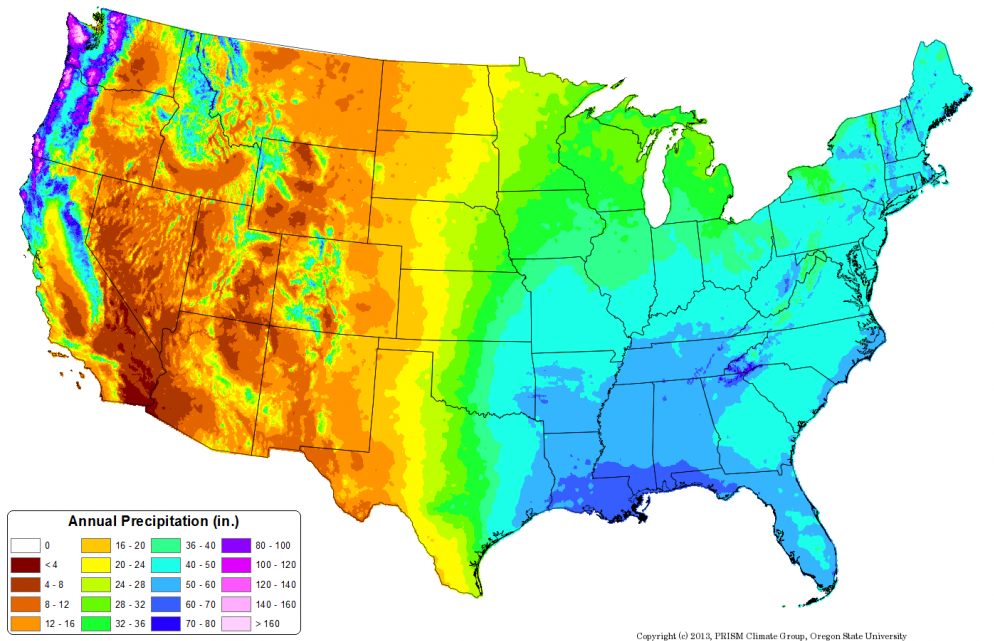}}
\end{tabular}
\caption{The contiguous United States graph. Vertices represent states, which are connected by an edge if they share a border.  The first two eigenvectors by frequency are ``smooth'' with respect to the graph geometry -- they vary smoothly across the country. The eleventh eigenvector does not.  Average annual precipitation by state varies relatively smoothly across geography.  Most of the information of this function resides in the low-frequency eigenvectors of the graph Laplacian.}
\label{fig: precip example}
\end{figure}

 We will denote the associated eigenvectors as $\phi_1,\ldots, \phi_n$ and interpret them as functions $\phi_i: V\to \RR$, where $\phi_i(v) $ is the value of the $i^{th}$ eigenvector at the vertex $v$.  In highly symmetric graphs, we would expect vertices to reflect graph symmetries, but not otherwise be distinguished. Therefore we focus on the case of equal weights, where a graphical design is as follows.

\begin{definition}
\emph{Given a finite, simple, unweighted, connected graph $G = (V,E)$, we say a subset $W \subset V$ \emph{integrates an eigenvector} $\phi$ of $L= AD^{-1}-I$ if $$ \frac{1}{|W|} \sum_{w \in W}  \phi(w) = \frac{1}{|V|}\sum_{v \in V} \phi(v).$$
We say a subset $W \subset V$ \emph{integrates an eigenspace} $\L$ of $L$ if $W$ integrates every eigenvector in a basis of $\L$.
A  \emph{$k$-graphical design} is $W \subset V$ such that $W$ integrates the first $k$ eigenspaces with respect to the frequency ordering.  If the context is clear, we may drop the word graphical and refer to $k$-designs.}
\end{definition}

%Analogously to quadrature rules on spheres, we can think of  $(1/|V|)\sum_{v \in V} \phi(v)$ as an integral of $\phi$ over the discrete domain $V$.  Since our goal is to approximate this ``integral'' with smaller subsets of points $W$, we continue to use the term integration as introduced in \cite{graphdesigns}.
\textcolor{cyan}{By linearity, if $W$ integrates every eigenvector in a basis of the eigenspace $\L$, then $W$ integrates every vector
in $\L$, and therefore every basis of $\L$. Thus the definition of integrating an eigenspace is unambiguous.} Given Definition 1.1, it is natural to ask how good a design can be, in the sense that a small number of vertices integrates a large number of eigenspaces. In order to quantify this, we define the following.

\begin{definition}
 \emph{Let $G = (V,E)$ be a finite, simple, connected, and unweighted graph with $m$ distinct eigenspaces ordered from low to high frequency as $\L_1 \leq... \leq \L_m$, where the eigenvalue of $\L_1$ has the lowest frequency and the eigenvalue of $\L_m$ has the highest frequency.  An \emph{optimal design} is $W \subset V$ integrating $\L_1,\ldots, \L_k$ with} $$\frac{|W|}{\sum_{i=1}^{k}\dim(\L_{i})} = \min\left\{ \frac{|W'|}{\sum_{i=1}^{k'}\dim(\L_{i})}: W' \subset V \text{ integrates } \L_1,\ldots, \L_{k'}\right\}.$$ 
 \emph{We define $\eff(W)$ to be the ratio on the left hand side of this equality.}
\end{definition}

\begin{figure}[h!]
     \centering
\begin{tikzpicture}[scale =.5]
\tikzstyle{bk}=[fill,circle, draw=black, inner sep=2 pt]
\tikzstyle{red}=[fill =red,circle, draw=red, inner sep=2 pt]
\tikzstyle{bl}=[fill=blue,circle, draw=blue, inner sep=2 pt]
 \foreach \y in {1,2,3,4,5,6,7,8,9}
        {\node at (40*\y +72:3) (a\y) [bk] {};}
\foreach \y in {1,2,3}
        {\node at (120*\y +30:1.2) (b\y) [bk] {};}
\node at (a8)  [red] {};
\node at (b3)  [red] {};
\node at (a3)  [red] {};
\node at (a4)  [red] {};
\draw[thick] (b1) -- (b2) -- (b3) -- (b1);
\draw[thick] (a1) -- (a2) -- (a3) -- (a4) -- (a5) -- (a6) -- (a7) -- (a8) -- (a9) --(a1);
\draw[thick] (a7) -- (a9);
\draw[thick] (a1) -- (a3);
\draw[thick] (a4) -- (a6);
\draw[thick] (b1) -- (a2);
\draw[thick] (b2) -- (a5);
\draw[thick] (b3) -- (a8);
\end{tikzpicture}
     \caption{The truncated tetrahedral graph is a 3-regular graph on 12 vertices. This graph has the following eigenvalues ordered by frequency, where the exponents represent multiplicity: $ 0^1,  (-5/3)^3,(-1/3)^3,(-4/3)^3,(-1)^2. $ The subset of red vertices integrates all eigenspaces but the eigenspace for $\l = -1$.  This is thus a 4-graphical design with efficacy $4/10$.}
     \label{fig: trunc tetra}
 \end{figure}
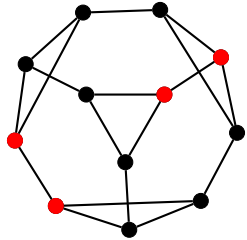
 
 \textcolor{cyan}{ The problem posed by Steinerberger \cite{graphdesigns} investigated how large $k$ can be given $|W|$. A dual question is, given $k$, how large must $|W|$ be?  The optimal design framework is distinct from both questions, as one could be interested in designs which are smaller than the optimal design or which integrate more eigenspaces than the optimal design. We find the notion of optimality a natural way to balance the trade off between $|W|$ and $k$.}

\subsection{Related work}

Quadrature rules for surfaces other than spheres are a relatively new field. For more on the smooth case, see, for instance, \cite{BRVmanifoldsI,BRVmanifoldsII,GGmanifolds, riemanndesigns}.
The graph sampling problem has been investigated primarily from an engineering perspective; see, for instance, \cite{samplingAGO,samplingTBD,samplingTEOC}. We also refer to the work of Pesenson (e.g. \cite{PesensonIII, PesensonI, PesensonII}).

The graphical design problem was first introduced by Steinerberger in \cite{graphdesigns}. Its main theorem loosely states that if $W$ is a good graphical design, then either $|W|$ is large, or the $j$-neighborhoods of $W$ grow exponentially. Steinerberger and Linderman (\cite{SSLinderman}) consider the numerical integration side of graphical designs.  They find an upper bound on the integration error for any quadrature rule on a graph, which multiplicatively separates out the size of a function $f$ and a quantity which can be interpreted as the quality of the quadrature scheme. 

Extremal designs, which are graphical designs that integrate all but one eigenvector, were introduced in \cite{Golubev}. Golubev shows that stable sets in $G=(V,E)$ which attain the Hoffman bound and subsets of $V$ which attain the Cheeger bound are extremal. He applies these results to several families of graphs, including Kneser graphs and graphs of the $d$-cube. We discuss this in more detail in Section 5.  
\vspace{.2 in}

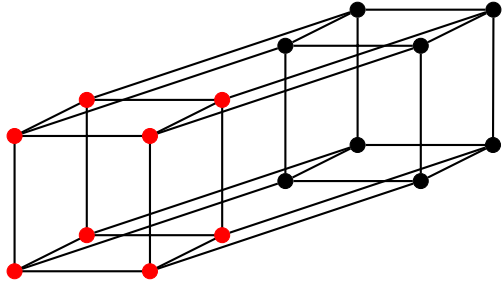
\begin{figure}[h!]
     \centering
\begin{tikzpicture}[scale =.6]
\tikzstyle{bk}=[fill,circle, draw=black, inner sep=2 pt]
\tikzstyle{red}=[fill =red,circle, draw=red, inner sep=2 pt]
\tikzstyle{bl}=[fill=blue,circle, draw=blue, inner sep=2 pt]
\draw[thick] (0,0) -- (3,0) -- (3,3) -- (0,3) -- (0,0) -- (1.6,.8) -- (4.6,.8) -- (4.6,3.8) -- (1.6,3.8) -- (1.6,.8);
\draw[thick] (3,0)  -- (4.6,.8);
\draw[thick] (3,3)  -- (4.6,3.8);
\draw[thick] (0,3)  -- (1.6,3.8);
\node (x1) at (0,0) [red] {};
\node (x2) at (3,0)  [red] {};
\node (x3) at (3,3)  [red] {};
\node (x4) at (0,3)  [red] {};
\node (y1) at (1.6,.8)  [red] {};
\node (y2) at (4.6,.8) [red] {};
\node (y3) at (4.6,3.8) [red] {};
\node (y4) at (1.6,3.8)  [red] {};
\node (v1) at (6,2)  [bk]{};
\node (v2) at (9,2) [bk] {};
\node (v3) at (6,5) [bk] {};
\node (v4) at (9,5) [bk] {};
\node (w1) at (7.6,2.8) [bk] {};
\node (w2) at (10.6,2.8) [bk] {};
\node (w3) at (10.61,5.8) [bk] {};
\node (w4) at (7.6,5.8) [bk] {};
\draw[thick] (v1) -- (v2) -- (v4) -- (v3) -- (v1) -- (w1) -- (w2) -- (w3) -- (w4) -- (w1);
\draw[thick] (v2)  -- (w2);
\draw[thick] (v4)  -- (w3);
\draw[thick] (v3)  -- (w4);
\draw[thick] (x1)  -- (v1);
\draw[thick] (x2)  -- (v2);
\draw[thick] (x3)  -- (v4);
\draw[thick] (x4)  -- (v3);
\draw[thick] (y1)  -- (w1);
\draw[thick] (y2)  -- (w2);
\draw[thick] (y3)  -- (w3);
\draw[thick] (y4)  -- (w4);
\end{tikzpicture}
     \caption{The vertices in red are a subset which attains the Cheeger bound for the cube graph $Q_4$. By  \cite[Theorem 2.4]{Golubev}, this is an extremal design. }
     \label{fig:Gol Q4}
 \end{figure}
\noindent

\section{Main Results}
We present an overview of the structure of this paper and the main results. In broad strokes, we summarize our findings with the following two principles:
\begin{principle}
Linear codes are a good place to look for graphical designs in cube graphs.  The Hamming code and some of its variants are effective graphical designs.
\end{principle}
\noindent
\begin{principle}
Graphical designs are distinct from several related and previously established combinatorial constructions on graphs.  Even in highly structured graphs, graphical designs are not the same as extremal designs, maximum stable sets in distance graphs, and $t$-designs on association schemes. 
\end{principle}
\noindent
Table \ref{table of designs} collects our main results and examples, and compares how they stack up among the variety of concepts we consider. We encourage the reader to revisit this table after encountering these combinatorial concepts in the text.

\subsection{Eigenspaces vs. Eigenvectors}
Graphical designs were originally defined in terms of integrating eigenvectors, not eigenspaces in \cite{graphdesigns}.  We show in Lemma \ref{change of basis} that a graphical design $W$ can either integrate an entire eigenspace $\L$, or if $W$ cannot integrate the eigenspace $\L$, then for any $i \in \{0,1,\ldots, \dim(\L)-1\}$, there is an eigenbasis $B$ of $\L$ such that $W$ integrates $i$ eigenvectors of $B$. The proof is constructive -- we provide what is essentially an algorithm for constructing $B$. Due to this lemma, we have made the choice to define graphical designs in terms of eigenspaces, not eigenvectors.

\subsection{Designs on Cube Graphs}

An open question in \cite{graphdesigns} was to find graphical designs on families of graphs.  To this end, Section 4 focuses on graphical designs on the graph of the $d$-cube, which we denote $Q_d$. We show that linear codes are generally good candidates [Theorem \ref{lincode thm}].  The Hamming code for the $(2^r-1)$-cube, in particular, integrates all but the last eigenspace in the frequency ordering [Theorem \ref{Main Hamm}].  Moreover, we show that the Hamming code is the unique smallest linear code which integrates all but the last eigenspace by frequency of the $(2^r-1)$-cube [Theorem \ref{Hamm is small}]. We can derive other highly effective graphical designs on \textcolor{cyan}{$Q_{2^{r}}$ and $Q_{2^{r}+1}$} by lifting the Hamming code [Theorem \ref{Lift Hamm}], and fairly effective graphical designs on \textcolor{cyan}{$Q_{2^r-2}$} by projecting the Hamming code [Theorem \ref{Hamm Proj}].

\subsection{Extremal Designs}
We next turn to the concept of extremal designs in Section 5. By definition, extremal designs do not consider the frequency ordering on eigenspaces, nor do they take into account the size of the graphical design. The main results in \cite{Golubev} find extremal designs through the Hoffman bound and Cheeger bound. We show that stable sets achieving the Hoffman bound and subsets acheiving the Cheeger bound do not integrate an eigenspace which is generally early in the frequency ordering [Theorems \ref{myHoff}, \ref{myCheeg}]. We show the Hamming code on the 3-cube is an optimal design by showing it is a stable set which meets the Hoffman bound for a different distance graph on the cube. We compare the efficacy of several graphical designs on cube graphs found here and in \cite{Golubev}. The Hoffman bound connects graphical designs to maximum stable sets. We show through an example that optimal designs need not be maximum stable sets, or stable sets at all.

\subsection{Association Schemes}
Another combinatorial object that seems, at first glance, to coincide with graphical designs are classical $t$-designs. We show that this is not the case in Section 6.  Classical $t$-designs do integrate some eigenvectors of $L=AD^{-1} -I$ [Theorem \ref{TD neq GD}]. However, optimal graphical designs need not be $t$-designs for any choice of $t$. We exhibit a family of graphs arising from the Johnson scheme which have an extremal design that is not a classical $t$-design for any choice of $t$ [Proposition \ref{johnson graph counterex}]. In particular, we show a graph arising from the $(8,3)$ Johnson scheme for which this subset is optimal (Example \ref{optimal but not t johnson}). We also show that the Hamming code is a better graphical design than $t$-design, integrating about twice as many eigenspaces in the frequency ordering as opposed to the ordering imposed in the association scheme framework [Proposition \ref{hamm better GD than TD}].

\begin{table}[h!]
    \centering
\begin{tabular}{ c c | c c c c c}
  Graph &  Subset & $k$-design & $t$-design & Extremal & Optimal & Stable Set\\  \hline
  $K_n$ & $\{1\}$ & 1 & N/A & yes & yes & yes \\
$Q_{2^r-1}$ & $H_{r}$ & $2^r-1$ &  $2^{r-1}-1$ & yes & ? & yes \\ 
$Q_{2^r-2}$ & $\pi(H_{r})$ & $2^r-3$ &  $2^{r-1}-2$& no& ? & yes \\
$Q_4$ & $H_2'$ & 4 &  1 & yes& yes & no \\
$Q_d$ &  $\{x: x_1 =0\}$ & $3$ & no & yes& no & no \\
$G_2$ & Y  & $3$ & no& yes & yes & no \\
\end{tabular}
    \caption{$Q_{d}$ is the $d$-cube graph. $H_{r}$ denotes the Hamming code on the $(2^r-1)$-cube, $\pi(H_r)$ denotes the projected Hamming code, and $H_r'$ denotes its lift. $G_2$ refers to the construction defined in Lemma \ref{distgraphs} for the (8,3) Johnson scheme, and $Y$ is the collection of all three-element subsets of $[8]$ containing $\{1\}$ (see Example \ref{optimal but not t johnson}). }
    \label{table of designs}
\end{table}
\vspace{-.2 in}

\section{The Problem of Eigenspaces vs. Eigenvectors}

Graphical designs were originally defined by ordered integration of eigenvectors, not eigenspaces, in \cite{graphdesigns}.  It was left undetermined how to handle eigenspaces with multiplicity in the discrete case. In the continuous case, multiplicity of eigenspaces is typically ignored due to the following folklore result. 
\begin{lemma}[Folklore]
\emph{For a smooth manifold, the dimension of an eigenspace of the Laplace-Beltrami operator $\Delta$ is small compared to the prior number of eigenvalues \textcolor{cyan}{(with multiplicity)} ordered by frequency.}
\end{lemma} 
\begin{proof}[Idea]
Weyl's law (\cite{Weyl}) states that for a sufficiently smooth compact manifold \textcolor{cyan}{embedded} in $\RR^d$, the number of eigenvalues $\leq \lambda$ of $\Delta$ \textcolor{cyan}{(with multiplicity)} is $$N(\lambda) = \b\cdot\lambda^{d/2}  + \cO(\lambda^{(d-1)/2})$$
for some constant $\b$.  So, suppose there is an eigenvalue $\l$ whose eigenspace $\L$ has multiplicity. Then for some small $\e$,
$N(\l+\e) - N(\l-\e) = \dim (\L).$
By applying Weyl's law twice, we have that $ \dim (\L) = \cO(\l^{(d-1)/2}).$
On the other hand, the total number of eigenvalues $<\l$ (with multiplicity) is on the order of $\l^{d/2}$, which is comparatively much bigger than $\dim (\L)$. 
\end{proof}

In the discrete case, the problem of eigenspaces with multiplicity is more substantial. There is ambiguity when confronted with infinitely many eigenbases, each of which can behave differently in terms of integration by a subset $W\subset V$. Some graphs come with a well-known set of eigenvectors; for example, the eigenvectors of a Cayley graph can be found using group characters. However, from a numerical integration perspective, there is no obvious reason for why such an eigenbasis is better than any other eigenbasis.  Moreover, the majority of graphs do not come paired with a well-known eigenbasis. The following lemma justifies Definitions 1.2 and 1.3 in terms of eigenspaces, by arguing that an eigenspace should be thought of as a single unit, rather than as a collection of eigenvectors. 

For $W\subset V$, let $\ones_W \in \RR^V$ be the indicator vector of $W$, that is $\ones_W(x) = 1$ if $x\in W$ and $\ones_W(x) =0$ if $x\notin W$. We denote the all-ones vector $\ones_V \in \RR^V$ by $\ones.$

\begin{lemma}\label{change of basis}
\emph{\textcolor{cyan}{Let $G=(V,E)$ have an eigenspace $\L$ of $L$ with $\dim \L >1$. For any $W \subset V$, either $W$ integrates $\L$, or for each $j =0, \ldots, \dim \L-1$, there is an orthonormal basis of $\L$ such that $W$ integrates precisely $j$ eigenvectors in this basis.} }
\end{lemma}
\begin{proof} \textcolor{cyan}{Suppose $W$ does not integrate $\L$. We will first show that if $W$ integrates precisely $j \in \{1,\ldots, \dim \L-1\}$ eigenvectors in a basis of $\L$, we can construct an orthonormal $B$ of $\L$ for which $W$ integrates precisely $j-1$ eigenvectors in $B$.  Let $\{\phi_i\}_{i=1}^{\dim \L} $ be an orthonormal basis for $\L$ where $W$ integrates $j$ eigenvectors of this basis.  We may reorder the basis vectors and suppose that $W$ integrates $\phi_2$ but not $\phi_1$. Consider $$B = \{ \phi_1 + \phi_2,   \phi_1 - \phi_2, \phi_3, \ldots, \phi_{\dim \L}\}.$$ Then $\spanset(B) = \L$, $$\langle \phi_1 \pm \phi_2, \phi_i \rangle =  \langle \phi_1 , \phi_i \rangle \pm  \langle \phi_2, \phi_i \rangle =0 \text{ for } 2 < i \leq \dim \L,  \text{ and }$$
 $$ \langle \phi_1 +\phi_2, \phi_1 -\phi_2 \rangle = \|\phi_1\|^2 - \|\phi_2\|^2 + \langle \phi_1 , \phi_2 \rangle- \langle \phi_1 , \phi_2 \rangle = 0.$$
  Thus $B$ is an orthogonal basis for $\L$.  Since $W$ does not integrate $\phi_1$ and does integrate $\phi_2$, we have
  $$\frac{1}{|W|}\phi_1^T \ones_W \neq \frac{1}{|V|} \phi_1^T \ones \text{   and } \frac{1}{|W|}\phi_2^T \ones_W = \frac{1}{|V|} \phi_2^T \ones .$$ Therefore,
  $$ 0 \neq \frac{1}{|W|}\phi_1^T \ones_W - \frac{1}{|V|}\phi_1^T \ones =  \frac{1}{|W|}(\phi_1 \pm \phi_2)^T \ones_W - \frac{1}{|V|}(\phi_1 \pm \phi_2)^T \ones .$$ 
    Thus $W$ integrates $j-1$ eigenvectors in $B$, and we can normalize to obtain an orthonormal basis of $\L$ with this property. }
  
\textcolor{cyan}{Next, we show that if $W$ integrates exactly $j \in \{0, \ldots, \dim \L-2\}$ eigenvectors in a basis of $\L$, we can construct an orthonormal $B$ of $\L$ for which $W$ integrates precisely $j+1$ eigenvectors in $B$. Let $\{\phi_i\}_{i=1}^{\dim \L} $ be an orthonormal basis for $\L$ where $W$ integrates $j$ vectors of $\{\phi_i\}_{i=1}^{\dim \L} $.  We may reorder the basis vectors and suppose that $W$ does not integrate $\phi_1$ and $\phi_2$.  Let $\a_1 = \phi_1^T( \ones/|V| - \ones_W/|W|)$ and   $\a_2 = \phi_2^T ( \ones/|V| - \ones_W/|W|)$. Since $W$ does not integrate $\phi_1$, $$\a_1 =\phi_1^T\left( \frac{\ones}{|V|} - \frac{\ones_W}{|W|}\right) \neq 0.$$  Likewise we have that $\a_2 \neq 0$. Consider $$B = \left\{\a_1\phi_1 + \a_2\phi_2,\frac{\phi_1}{\a_1} - \frac{\phi_2}{\a_2}, \phi_3, \ldots, \phi_{\dim \L}\right\} .$$  Then $\spanset(B) = \L$,
$$\langle  \a_1\phi_1 + \a_2\phi_2, \phi_i \rangle =0  \text{ and } \left\langle  \frac{\phi_1}{\a_1} - \frac{\phi_2}{\a_2}, \phi_i \right\rangle =0 \text{ for } 2 < i \leq \dim \L, \text{ and}$$
$$ \left\langle \a_1\phi_1 + \a_2\phi_2,\frac{\phi_1}{\a_1} - \frac{\phi_2}{\a_2} \right\rangle = \frac{\a_1}{\a_1}
\|\phi_1\|^2 + \frac{\a_2}{\a_1} \phi_2^T\phi_1 - \frac{\a_1}{\a_2} \phi_1^T\phi_2-  \frac{\a_2}{\a_2}\|\phi_2\|^2 =0,$$
so $B$ is an orthogonal eigenbasis for $\L$. Moreover, $W$ integrates $\phi_1/\a_1 - \phi_2/\a_2$:
$$\frac{1}{|V|} \left(\frac{\phi_1}{\a_1} - \frac{\phi_2}{\a_2}\right)^T \ones - \frac{1}{|W|}\left(\frac{\phi_1}{\a_1} - \frac{\phi_2}{\a_2}\right)^T \ones_W  =   \frac{\a_1}{\a_1} - \frac{\a_2}{\a_2}  = 0.$$ However $W$ does not integrate $\a_1\phi_1 +\a_2 \phi_2$:
$$ \frac{1}{|V|}(\a_1\phi_1 + \a_2\phi_2)^T \ones-  \frac{1}{|W|}(\a_1\phi_1 + \a_2\phi_2)^T \ones_W  =  \a_1^2 +\a_2^2 \neq 0.$$
 Thus $B$ is a orthogonal basis for $\L$ such that $W$ integrates precisely $j+1$ eigenvectors in $B$. We can normalize to obtain an orthonormal basis with this property.}
 
\textcolor{cyan}{Thus if $W$ does not integrate a given a basis $\{\phi_i\}_{i=1}^{\dim \L}$ of $\L$ entirely, we can iterate one of these two processes to find orthonormal eigenbases for $\L$ such that $W$ integrates $j$ eigenvectors for $j = 0,\ldots, \dim \L-1$.}
\end{proof}
In light of this lemma, we conclude that Definitions 1.2 and 1.3, which define $k$-graphical designs, optimal designs, and efficacy in terms of eigenspaces, are sensible. There is too much flexibility when choosing the basis of an eigenspace.
\begin{example} We illustrate eigenspace multiplicity with $K_5$, the complete graph on 5 vertices. For $K_5$, $L$ has only one nontrivial 4-dimensional eigenspace $\L= \{ x: \ones^T x =0\}$. The subset $\{1\}$ does not integrate $\L$.  By Lemma \ref{change of basis}, there are bases of $\L$ such that $\{1\}$ integrates $0,1,2$ and $3$ eigenvectors in the basis, which we exhibit below. The columns of the following matrices form orthogonal bases for $\L$.  The subset $\{1\}$ integrates an eigenvector $\phi\in \L$ if and only if $\phi(1)=0$.  For $i =0,\ldots, 3$, $\{1\} $ integrates precisely the first $i$ columns of $U_i$.  
 \begin{align*}
&U_0 = \begin{bmatrix}  
   -2 & -1 & -1& -1\\
    \phantom{-}1 & -2& \phantom{-}5& -1\\
   \phantom{-}1& \phantom{-}0&- 7& -1\\
   \phantom{-}0 & \phantom{-}3&\phantom{-}3 & -1\\
    \phantom{-}0& \phantom{-}0& \phantom{-}0& \phantom{-}4\\\end{bmatrix}, 
U_1 = \begin{bmatrix} 
   \phantom{-}0 & \phantom{-}1& -1  &\phantom{-} 1  \\
   \phantom{-}0 &\phantom{-}1 &\phantom{-} 1 & \phantom{-}1   \\
  \phantom{-}0 & \phantom{-}0& \phantom{-} 0& -4 \\
    \phantom{-}1&-1 & \phantom{-}0 &\phantom{-}1  \\
   -1 &  -1& \phantom{-}0 & \phantom{-}1 \\\end{bmatrix}, \\
&U_2 = \begin{bmatrix}  
     \phantom{-}0&\phantom{-}0 & -2 & -2 \\
    \phantom{-}0& -1&\phantom{-}1 & -2\\
    \phantom{-}0& \phantom{-}1 & \phantom{-}1& -2\\
  \phantom{-}1&\phantom{-}0 & \phantom{-}0& \phantom{-}3\\
   -1& \phantom{-}0& \phantom{-}0& \phantom{-}3\\\end{bmatrix}, 
U_3 = \begin{bmatrix}  
   \phantom{-}0& \phantom{-}0 & \phantom{-}0& -4 \\
  \phantom{-}0& -1& -1& \phantom{-}1 \\
   \phantom{-}0& \phantom{-}1 & -1&  \phantom{-}1\\
   \phantom{-}1&\phantom{-}0 &\phantom{-}1 & \phantom{-} 1\\
    -1& \phantom{-}0& \phantom{-}1& \phantom{-} 1\\\end{bmatrix}.
    \end{align*}

\end{example}

\section{The $d$-Dimensional Cube}

In this section we show that linear codes can provide effective graphical designs on the graphs of cubes, and that constructions from the Hamming code are particularly effective. We denote by $Q_d$ the graph of the $d$-cube, which has vertex set $\{0,1\}^d$ and an edge between vertices $v$ and $w$ if they differ in exactly one coordinate. Later we will build graphs on $\{0,1\}^d$ with other types of edges. 
 
 \begin{figure}[h!]
\begin{align*}
& \begin{tikzpicture}[baseline={($ (current bounding box.west) - (0,1ex) $)}, scale = .7]
%% 1 dim cube 
%\draw[thick] (-7,0.5) --  (-7,3.5);
%% vertices
% \draw  (-7,0.5) node[circle, fill = white,inner sep=2pt,draw] {0};
% \draw  (-7,3.5) node[circle, fill = white,inner sep=2pt,draw] {1};
%% labels
%% 2 dim cube 
%\draw[thick] (-5,0.5) -- (-2,0.5) -- (-2,3.5) -- (-5,3.5) -- (-5,0.5); %% vertices
%\draw  (-5,.5)  node[circle, fill = white,inner sep=2pt,draw] {00};
%\draw  (-2,.5)  node[circle, fill = white,inner sep=2pt,draw] {10};
%\draw  (-5,3.5)  node[circle, fill = white,inner sep=2pt,draw] {01};
%\draw  (-2,3.5)  node[circle, fill = white,inner sep=2pt,draw] {11};
%% 3 dim cube
%%% edges i =1
\draw[thick] (0,0) -- (3,0) -- (3,3) -- (0,3) -- (0,0) -- (1.6,.8) -- (4.6,.8) -- (4.6,3.8) -- (1.6,3.8) -- (1.6,.8);
\draw[thick] (3,0)  -- (4.6,.8);
\draw[thick] (3,3)  -- (4.6,3.8);
\draw[thick] (0,3)  -- (1.6,3.8);
%% vertices
\draw  (0,0)  node[circle, fill = white,inner sep=2pt,draw] {000};
\draw  (3,0)  node[circle, fill = white,inner sep=2pt,draw] {100};
\draw  (0,3)  node[circle, fill = white,inner sep=2pt,draw] {010};
\draw  (3,3)  node[circle, fill = white,inner sep=2pt,draw] {110};
\draw  (1.6,.8)  node[circle, fill = white,inner sep=2pt,draw] {001};
\draw  (4.6,.8)  node[circle, fill = white,inner sep=2pt,draw] {101};
\draw  (1.6,3.8)  node[circle, fill = white,inner sep=2pt,draw] {011};
\draw  (4.6,3.8)  node[circle, fill = white,inner sep=2pt,draw] {111};
\end{tikzpicture}
&& \begin{bmatrix}  
1 & \phantom{-}1 & \phantom{-}1 & \phantom{-}1& \phantom{-}1 & \phantom{-}1 & \phantom{-}1 & \phantom{-}1 \\
1 & -1 & \phantom{-}1 & \phantom{-}1& -1 & -1 & \phantom{-}1 & -1 \\
1 & \phantom{-}1 & -1 & \phantom{-}1& -1 & \phantom{-}1 & -1 & -1 \\
1 & \phantom{-}1 & \phantom{-}1 & -1& \phantom{-}1 & -1 & -1 & -1 \\
1 & -1 & -1 &\phantom{-}1&\phantom{-}1 & -1 & -1 & \phantom{-}1 \\
1 & -1 & \phantom{-}1 &-1& -1 & \phantom{-}1 & -1 & \phantom{-}1 \\
1 & \phantom{-}1 & -1 & -1& -1 & -1 & \phantom{-}1 & \phantom{-}1 \\
1 & -1 &-1 & -1& \phantom{-}1 & \phantom{-}1 & \phantom{-}1 &-1 
    \end{bmatrix}
\end{align*}
\caption{\textcolor{cyan}{The graph $Q_3$ and a matrix of its eigenvectors for $L= AD^{-1} - I$. The $i$-th row and $i$-th column are indexed by the vertex which is the binary expansion of $i-1$. The first column spans $\L_0$ with eigenvalue 0, columns 2, 3, and 4 span $\L_1$ with eigenvalue $-2/3$, columns 5, 6, and 7 span $\L_2$ with eigenvalue $-4/3$, and column 8 spans $\L_3$ with eigenvalue $-2$. }}
\label{fig:cubes}
\end{figure}
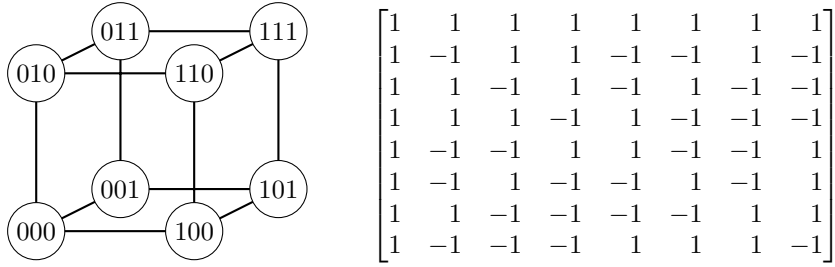

The graph $Q_d$ is a Cayley graph, where the underlying group  is $\{0,1\}^d$ with addition mod 2, and the generating set is $\{e_i\}_{i=1}^d$, the standard basis vectors. For more on Cayley graphs and their spectra, see \cite[Chapter 6]{BrouwerHaemers}; we also provide a brief discussion before Theorem \ref{H2 optimal}. The graph $Q_0$ is trivial and $Q_1$ is non-regular, so we consider $d \geq 2$, where $Q_d$ is $d$-regular. From the Cayley graph structure, we can derive the eigenvectors and eigenvalues of $Q_d$ using the group characters of $\{0,1\}^d$. For more on the representation of finite groups, we recommend \cite[Chapter 1]{Sagan}.  Each $v \in \{0,1\}^d$ gives an eigenvector $\phi_v(x) = (-1)^{v^Tx}$ with eigenvalue $(1/d) \sum_{i=1}^d  (-1)^{v_i} -1.$ To determine the frequency ordering, we introduce the Hamming distance.

\begin{definition}[see \cite{ECC}]
\emph{The\emph{ Hamming distance} $d_H$ between two vectors $v$ and $w$ in $\{0,1\}^d$ is the number of coordinates they differ on. That is,
$$ d_H(v,w) =  \sum_{i=1}^d | v_i-w_i |.$$
The\emph{ weight }of $v$ is $wt(v) = d_H(v,0) = \ones^T v$, which is the number of ones in $v$.}
\end{definition}
\textcolor{cyan}{ To better reflect the structure of the spectrum of $Q_d$, we make a slight change of notation. We have distinct eigenspaces $\L_i =\spanset \{ \phi_v: wt(v) = i\}$ with multiplicity $\genfrac(){0 pt}{1}{d}{i}$ and eigenvalue $-2i/d$ for each $i = 0,\ldots,d$. The ordering of eigenspaces by frequency for $d$ odd is then 
$$\L_0 \equiv \L_d <  \ldots <\L_i \equiv \L_{d-i} <  \ldots < \L_{\lceil d/2 \rceil} \equiv \L_{\lceil d/2 \rceil -1},$$
where $\L_i<\L_j$ denotes that the eigenvalue of $\L_i$ has a strictly lower frequency than that of $\L_j,$ and $\L_i\equiv
\L_j$ denotes eigenspaces which are interchangeable in the frequency ordering.  If $d$ is even, the eigenspace ordering is
$$\L_0 \equiv \L_d <  \ldots <\L_i \equiv \L_{d-i} <  \ldots < \L_{d/2} .$$}
 See Figure \ref{fig:cubes} for $Q_3$ and its eigenspaces.

\begin{subsection}{Preliminary Results }

We first provide some simple graphical designs on $Q_d$.  A regular graph has a trivial one-dimensional eigenspace with eigenvalue $0$ spanned by $\ones $.  This agrees with the Cayley graph structure: $\phi_0(x) = (-1)^{x^T0} = 1$ for all $x\in \{0,1\}^d$.
Any $ W \subset \{0,1\}^d$ integrates $\ones$. We will often use the following fact.
\begin{lemma}\label{integrate reg means =0}
\emph{Let $G=(V,E)$ be regular. A subset $W\subset V$ integrates a nontrivial eigenvector $\phi$ \textcolor{cyan}{if and only} if $\phi^T \ones_W =0$.}
\end{lemma}
\begin{proof}
Since $G$ is regular, $\spanset \{\ones\}$ is an eigenspace of $L$.
Since eigenspaces are orthogonal, every nontrivial eigenvector $\phi$ is such that $\phi^T \ones =0$. \textcolor{cyan}{Thus $W$ integrates $\phi$ if and only if $$ \frac{1}{|W|}\phi^T\ones_W =\frac{1}{|W|}\sum_{w\in W} \phi(w) = \frac{1}{|V|}\sum_{v\in V} \phi(v) =\frac{1}{|V|}\phi^T \ones =0.$$}
\end{proof}
\noindent
The first nontrivial eigenspace by frequency is $\L_d=\spanset\{\phi_\ones\}$ with eigenvalue $-2$.
\begin{lemma}
\emph{A subset $W\subset \{0,1\}^d$ integrates the eigenspace $\L_d=\spanset\{\phi_\ones\}$ of $Q_d$ if and only if it has an equal number of even and odd weight vertices.} \label{cube lambda n}
\end{lemma}

\begin{proof}
A subset $W\subset\{0,1\}^d$ satisfies $$0=\phi_\ones^T\ones_W = \sum_{w\in W } (-1)^{\ones^T w} =   \sum_{\substack{w\in W \\ \ones^Tw \text{ is even} }} 1 - \sum_{\substack{ w\in W \\ \ones^Tw \text{ is odd}}} 1$$ if and only if the number of even and odd weight vertices in $W$ are equal.
\end{proof}
The next eigenspace by frequency is $\Lambda_1 = \spanset\{\phi_{e_i}\}_{i=1}^{d}$. The following lemma addresses all odd eigenspaces. We abbreviate the subset $\{1,\ldots, d\}$ as $[d]$.

\begin{lemma}
\emph{If $W\subset \{0,1\}^d $ is such that $w\in W \iff \ones -w \in W$ and $i \in [d]$ is odd, then $W$ integrates the eigenspace $\Lambda_i$ of $Q_d$.}\label{cube odd eigspaces}
\end{lemma}
\begin{proof}
 Let $I\subseteq [d]$ be an index set of odd size, and let $e_I= \sum_{i\in I} e_i$ be its indicator vector. To integrate $\phi_{e_I} $, we need to show that
  $$0 = \phi_{e_I}^T \ones_W = \sum_{w\in W} (-1)^{e_I^T w}  = \sum_{\substack{w\in W \\ e_I^T w \text{ is even}}} 1 -  \sum_{\substack{w\in W \\ e_I^T w \text{ is odd}}} 1. $$
 It then suffices to find a bijection between the sets
 $$\left\{w\in W:e_I^T w \text{ is even}\right\} \text{ and } \left\{w\in W:e_I^T w\text{ is odd} \right\} .$$
If $|I|$ is odd, then $e_I^T (\ones -w) \mod 2 = 1 - e_I^T w  \mod 2,$
and so $e_I^T w$ and $e_I^T (\ones -w)$ have opposite parity. Thus mapping $w \mapsto \ones - w$ provides the desired bijection.
\end{proof}
As a consequence, for any $d$, we can find some very small graphical designs which are moderately effective.  More precisely,

\begin{lemma} \label{simple cube designs}
\emph{ If $d$ is odd, the subset $\{0, \ones\}$ is a 3-design, and $$ \eff(\{0,\ones\})= 2/(d+2) .$$  If $d>2$ is even, a subset such as $W = \left\{e_1, e_1+e_2, \ones - e_1 , \ones - e_1 -e_2\right\}$ is at least a 4-design, and $$ \eff(W) \leq 4/(2d+2).$$  \textcolor{cyan}{Moreover, a 4-design must contain at least 4 vertices for any value of $d$.}}
\end{lemma}

\begin{proof}
\textcolor{cyan}{Recall that $\dim{\L_0} = \dim{\L_d} =1$, and $\dim{\L_1} = \dim{\L_{d-1}} =d$. By Lemma \ref{cube lambda n}, a design which integrates $\L_d$ must have an even number of vertices. We claim that a subset of two vertices integrates $\L_1$ if and only if it is of the form $\{v, \ones -v\}$. Lemma \ref{cube odd eigspaces} shows one direction of this statement. Suppose that $\{v,w\} $ integrates $\L_1$.  Then for each $i=1,\ldots, d$, 
$$ 0 = \phi_{e_i}(v) + \phi_{e_i}(w) = (-1)^{v_i} + (-1)^{w_i}.$$ Hence $w_i = 1-v_i$ for each $i$, and so $w = \ones - v$. This ends the proof of the claim.}

\textcolor{cyan}{Suppose $d>1$ is odd.  Then $\{v,\ones-v\}$ integrates $\L_0, \L_d,$ and $\L_1$ by Lemma \ref{cube odd eigspaces}. Additionally, since $wt(\ones - e_i)$ is even, $$\phi_{\ones - e_i}(\ones - v) = \phi_{\ones - e_i}(\ones) \phi_{\ones - e_i}(v) =\phi_{\ones - e_i}(v).$$ Hence $\phi_{\ones - e_i}(v) + \phi_{\ones - e_i}(\ones - v) = 2\phi_{\ones - e_i}(v) \neq 0$ for any $v$, and so $\{v,\ones-v\}$ cannot integrate $\L_{d-1}$. Thus a 4-graphical design on $Q_d$ must contain at least 4 vertices. }

\textcolor{cyan}{Now, let $d>2$ be even. Then $v$ and $\ones-v$ have the same parity and hence cannot integrate $\L_d$.  Thus a design which integrates $\L_0, \L_1,$ and $\L_d$ contains at least 4 vertices.  An example of a minimal subset which integrates $\L_0, \L_1,$ and $\L_d$ is  $W = \left\{e_1, e_1+e_2, \ones - e_1 , \ones - e_1 -e_2\right\}$. $W$ has two odd weight vectors, two even weight vectors, and $w\in W \iff \ones -w \in W$.  Thus by Lemmas \ref{cube lambda n} and \ref{cube odd eigspaces}, $W$ integrates $\Lambda_0, \Lambda_d, $ and $\Lambda_1,$ as well as  $\Lambda_{d-1}$ since $d-1$ is odd.}
\end{proof}

We note that these subsets remain fixed no matter how large $d$ is.
\end{subsection}

\begin{subsection}{Linear Codes as Designs}
 We begin with a little background on codes.  For a more complete exposition, see Chapter 1 of \cite{ECC}, for instance.  A \emph{code} is a subset $C \subset \{0,1\}^d$, where $\{0,1\}^d$ is the set of all \emph{words}.  An element $c\in C$ is called a \emph{codeword}. We will stick with the case of transmitting bits, but one can also consider codes on a larger ``alphabet" than $\{0,1\}$.  The \emph{distance} of a code is $$\text{dist}(C) = \underset{c\neq c' \in C}{\min} d_H(c,c') .$$  

A \emph{linear code} is a linear subspace of $\{0,1\}^d$, viewed as a vector space with addition mod 2. It is quick to check that the distance of a linear code is the minimum weight of a nonzero codeword. Every linear code $C$ can be described as the kernel of a \emph{check matrix} $M$:  $$C = \{v \in \{0,1\}^d: Mv =0\}.$$   If $M$ is the $K \times d$ check matrix for a code $C$, then $C$ is an $(d-K)$-dimensional vector space over $\{0,1\}$, and the row span of $M$ is a $K$-dimensional vector space over $\{0,1\}$. We will also need the concept of dual codes.  

\begin{definition}[see Section 1.8 of \cite{ECC}]
\emph{Let $C \subset \{0,1\}^d$ be a linear code with check matrix $M$. The \emph{dual code} of $C$ is  $$C^\perp = \{ v \in \{0,1\}^d: v^Tc =0 : \forall c\in C\}. $$  In other words, $C^\perp$ is the row span of $M$ in $\{0,1\}^d$.}
\end{definition}

 A good code is both spread out among the vertices of $Q_d$ and in some sense near all the other vertices, which aligns with what we  desire from a graphical design. So, we investigate linear codes as designs on $Q_d$. We will use the following technique.

\begin{lemma}[\cite{Golubev}]
\emph{\textcolor{cyan}{Let $G=(V,E)$ be regular and let $\{\phi_1 = \ones,\ldots, \phi_n\}$ be an orthogonal basis of eigenvectors for $L$. If $W\subset V$ is such that $\ones_W$ can be written as a linear combination of the eigenvectors $\ones, \phi_{j_1},\ldots, \phi_{j_N}$, then $W$ integrates the eigenvectors other than $\phi_{j_1},\ldots, \phi_{j_N}$.}} \label{proof technique}
\end{lemma}
\begin{proof}
Suppose $\ones_W = \a_0\ones + \sum_{i=1}^N \a_i\phi_{j_i}.$  Let $j \notin \{1, j_1,\ldots, j_N\}$.  Since $G$ is regular, Lemma \ref{integrate reg means =0} gives us that $W$ integrates $\phi_j$ if $\phi_j^T\ones_W=0$. Then 
$$ \phi_j^T \ones_W = \phi_j^T \left(\a_0\ones + \sum_{i=1}^N \a_i\phi_{j_i}\right)=0, $$
by orthogonality. Thus $W$ integrates $\phi_j$.
\end{proof}

Note that being able to express $\ones_W$ as $ 
\a_0\ones + \sum \a_i\phi_{j_i}$ does not necessarily preclude $W$ from integrating some of $\{\phi_{j_i}\}_{i=1}^N$. 

\begin{theorem}
\emph{\textcolor{cyan}{Let $C = \{x: Mx =0\}$ be a linear code in $\{0,1\}^d$, where $M$ is a $K \times d$ matrix. $C$ integrates $\phi_v$ if and only if $v \notin C^{\perp}$ or $v$ is the zero vector. }
\label{lincode thm}}
\end{theorem}

\begin{proof}
    Let $C$ be a linear code with check matrix $M$, which has rows $a_i$, $i =1,\ldots, K.$  For $I \subseteq [K]$, denote $a_I =\sum_{i\in I} a_i.$  We take $a_\vn = 0$, the all-zeros vector. \textcolor{cyan}{Since $C^\perp$ is the rowspan of $M$, $C^{\perp}= \{a_I: I\subseteq K\}.$} By Lemma \ref{proof technique},  showing $$\ones_C = \frac{1}{2^K} \sum_{I\subseteq [K]} \phi_{a_I} $$ proves that if $v \notin C^\perp$, then $C$ integrates $\phi_v$. 
Let $c\in C$. Then $a_i^Tc =0$ for each $i=1,\ldots, K$.  Hence $\phi_{a_I}(c)= 1 \text{ for each } I\subseteq [K].$  There are $2^K$ subsets of $[K]$, so $$\frac{1}{2^K} \sum_{I\subseteq [K]} \phi_{a_I}(c) =1. $$
Let $v\notin C$.  Then there is a row $a_j$ such that $a_j^Tv =1$. \textcolor{cyan}{Then, consider all subsets $I\subseteq [K]$ such that $j\notin I$. We see that}
$$\phi_{a_{I \cup j}}(v) =  \phi_{a_j + a_I}(v) = \phi_{a_j}(v)\phi_{a_I}(v) = -\phi_{a_I}(v).$$  
Therefore, \textcolor{cyan}{by pairing off the eigenvectors $\phi_{a_I}$ and $\phi_{a_{I \cup j}}$, we have}
$$ \sum_{I\subseteq [K]} \phi_{a_I}(v) =  \sum_{\substack{I\subseteq [K] \\ j \notin I}} \left(\phi_{a_I}(v) + \phi_{a_{I \cup j}}(v)\right) = 0.$$

\textcolor{cyan}{To show the other direction, we first note that the eigenvectors $\phi_v$ are orthogonal. If $v,w \in\{0,1\}^d$ are distinct, then $$\phi_v^T\phi_w = \sum_{x\in\{0,1\}^d}  \phi_v(x) \phi_w(x) =\sum_{x\in\{0,1\}^d}  \phi_{v+w}(x) = \ones^T\phi_{v+w} = 0$$ since all nontrivial eigenspaces are orthogonal to $\ones$.  Now, suppose $v \in C^\perp$ is nonzero. Then, $v = a_{J}$ for some $\vn \neq J \subseteq [K]$.  By orthogonality, we see that
\begin{align*}
    \phi_{a_{J}}^T \ones_C =   \phi_{a_{J}}^T \left(\frac{1}{2^K} \sum_{I\subseteq [K]} \phi_{a_I} \right)= \frac{1}{2^K}  \phi_{a_{J}}^T \phi_{a_{J}} = \frac{1}{2^K} \|\phi_{a_{J}}\|^2 \neq 0.
\end{align*}
Hence $C$ does not integrate $\phi_v$.}
\end{proof}
The dual code may contain eigenvectors of any weight, so the unintegrable eigenvectors may come anywhere in the frequency ordering. Thus we seek linear codes where all rows of $M$ are near $d/2$ in weight, so that the unintegrable eigenspaces are as far out in the frequency ordering as possible.  
\end{subsection}

\begin{subsection}{Hamming Codes}
In this section, we look at Hamming codes and extensions of them as designs. Hamming codes are linear codes first introduced by Hamming in \cite{HammingOG} which are built on the idea of parity checking.  For each $r\geq 2$, there is a Hamming code $H_r$ in the $(2^r -1)$-cube.   $H_r$ is a vector subspace of dimension $2^r-r-1$, and $\dist(H_r) =3$. The cardinality of $H_r $ is $2^{2^r-r-1}$. We again refer the reader to  Chapter 1 of \cite{ECC} for a more complete introduction.

The check matrix $M_r$ of the Hamming code $H_r$ is the $r \times (2^r -1)$ matrix whose columns are the nonzero vertices of the $r$-cube.  We can see this in the check matrix \[M_3= \begin{bmatrix} 
1& 0 & 1 & 0& 1 &0& 1 \\
 0 & 1 & 1 & 0 & 0 & 1 & 1 \\
 0 & 0 & 0 & 1 & 1 & 1 & 1 \\
\end{bmatrix}. \]
The dual code $H_r^\perp$, which we recall is the rowspan of $M_r$, is called the simplex code, so named because its vectors form the vertex set of a regular $(2^r-1)$-simplex. We will use the following facts about the simplex code (see \cite[Section 1.9]{ECC}).
The simplex code in dimension $2^{r}-1$ consists of the zero vector and $2^r-1$ vectors of weight $2^{r-1}$.  For all $v,w\in H_r^\perp,$ $d_H(v,w) = 2^{r-1}$.

\begin{theorem}\label{Main Hamm}
\emph{The Hamming code $H_r$ is a $(2^r-1)$-design on $Q_{2^r-1}$, and }
$$\eff(H_r) = \frac{2^{2^r-r-1}}{2^{2^r-1}- \genfrac{(}{)}{0pt}{1}{2^{r}-1}{2^{r-1}}} \sim \frac{1}{2^r} \text{ asymptotically}.$$
\end{theorem}
\begin{proof}
 All nonzero vectors \textcolor{cyan}{in $H_r^\perp$ have weight $2^{r-1}$. By Theorem \ref{lincode thm}, any eigenvector which $H_r$ cannot integrate must lie in the eigenspace $\Lambda_{2^{r-1}}$.  Thus $H_r$} integrates all eigenspaces besides $\Lambda_{2^{r-1}}$. Recall that the frequency ordering of $Q_{2^r-1} $ is
$$\L_0 \equiv \L_{2^r-1}< \L_1 \equiv \L_{2^r-2} < \ldots < \L_{2^{r-1}-1}\equiv \L_{2^{r-1}} .$$
 Thus $\Lambda_{2^{r-1}}$ can be ordered last by frequency in $Q_{2^r-1}$.  Recall that the numerator of $\eff(H_r)$ is $|H_r| = 2^{2^r-r-1}$, and since $\dim(\L_{2^{r-1}})=\binom{2^{r}-1}{2^{r-1}}$, we have that $$\sum_{i\neq 2^{r-1}} \dim(\L_i)=2^{2^r-1}- \binom{2^{r}-1}{2^{r-1}} .$$ 
Asymptotically, $$\binom{2^{r}-1}{2^{r-1}}< \binom{2^{r}}{2^{r-1}} \sim \frac{4^{2^{r-1}}}{\sqrt{\pi2^{r-1}}} \ll 2^{2^r-1},$$
thus as $r$ grows, $$\eff(H_r) \sim \frac{2^{2^r-r-1}}{2^{2^r-1}} = \frac{1}{2^r}.$$ 
\end{proof}

\begin{theorem}
\emph{The Hamming code $H_r$ is the smallest linear code in cardinality which integrates all eigenspaces of $Q_{2^r-1}$ except for the eigenspace $\Lambda_{2^{r-1}}$. } \label{Hamm is small}
\end{theorem}

\begin{proof}
Let $M$ be the $K \times (2^r -1)$ check matrix of an arbitrary linear code $C$.  If $K < r$, then $|C| > |H_r|$. 

Suppose $K =r$ and $C$ integrates all eigenspaces except for $\L_{2^{r-1}}$. Then $C^\perp$, the row span of $M$, consists of $2^r$ vectors of weight $2^{r-1}$.  Since $C^\perp$ is linear, it follows that $C^\perp$ consists of $2^r$ vectors all at distance $2^{r-1}$ from each other. It is a standard result in discrete geometry that $2^r$ equally spaced vertices in dimension $2^r-1$ must form the vertex set of a regular $(2^r-1)$-simplex. Therefore, $C^\perp$ is the simplex code, and so $C =H_r$, up to relabeling of the vertices.

If $K>r,$ then $C^\perp$ consists of more than $2^r$ points. Another standard result in discrete geometry is that the maximum number of equidistant points in $\RR^{2^r-1}$ is $2^r$.  Hence there must be nonzero vectors $u$ and $v$ in $C^\perp$ such that $ wt(u) \neq wt(v)$. Thus $C$ cannot integrate the distinct eigenspaces $\L_{wt(u)}$ and $\L_{wt(v)} $ by Lemma \ref{lincode thm}.
\end{proof}

We show in Theorem \ref{H2 optimal} that the Hamming code $H_2$ is optimal in the sense of Definition 1.3 on $Q_3$, but it is unknown whether any other Hamming codes are optimal. Recall that in $Q_d$, $\dim (\L_j) = \binom{d}{j}$. In light of Theorem \ref{Hamm is small} and the increasing dimension of higher frequency eigenspaces, we suspect the following.
\begin{conjecture}
\emph{The Hamming code $H_r$ is an optimal design on $Q_{2^{r}-1}$.}
\end{conjecture}

\end{subsection}

\begin{subsection}{Extensions of Hamming Codes}

Variations of the Hamming code provide effective designs on other cube graphs. We start by lifting $H_r$ to higher dimensions.  Let $M_r$ be the check matrix of the Hamming code $H_r$, and $H_r'$ be the code with check matrix $M_r' = [M_r \ 0]$, where $0$ represents a column of zeros. For instance, for $r=3$, this would yield a linear code on the 8-cube with check matrix
\[M_3'= \begin{bmatrix}
1 & 0 & 1 & 0 & 1 & 0 & 1 & 0 \\
0 & 1 & 1 & 0 & 0 & 1 & 1 & 0 \\
0 & 0 & 0 & 1 & 1 & 1 & 1 & 0 \\
\end{bmatrix}.\]
We can do this again.  Let $H_r''$  be the code with check matrix $M_r'' = [M_r \ 0 \ 0]$.

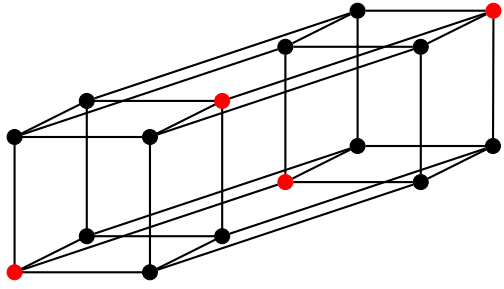
\begin{figure}[h!]
     \centering
\begin{tikzpicture}[scale =.6]
\tikzstyle{bk}=[fill,circle, draw=black, inner sep=2 pt]
\tikzstyle{red}=[fill =red,circle, draw=red, inner sep=2 pt]
\tikzstyle{bl}=[fill=blue,circle, draw=blue, inner sep=2 pt]
\draw[thick] (0,0) -- (3,0) -- (3,3) -- (0,3) -- (0,0) -- (1.6,.8) -- (4.6,.8) -- (4.6,3.8) -- (1.6,3.8) -- (1.6,.8);
\draw[thick] (3,0)  -- (4.6,.8);
\draw[thick] (3,3)  -- (4.6,3.8);
\draw[thick] (0,3)  -- (1.6,3.8);
\node (x1) at (0,0) [red] {};
\node (x2) at (3,0) [bk] {};
\node (x3) at (3,3) [bk] {};
\node (x4) at (0,3) [bk] {};
\node (y1) at (1.6,.8) [bk] {};
\node (y2) at (4.6,.8) [bk] {};
\node (y3) at (4.6,3.8) [red] {};
\node (y4) at (1.6,3.8) [bk] {};
\node (v1) at (6,2) [red] {};
\node (v2) at (9,2) [bk] {};
\node (v3) at (6,5) [bk] {};
\node (v4) at (9,5) [bk] {};
\node (w1) at (7.6,2.8) [bk] {};
\node (w2) at (10.6,2.8) [bk] {};
\node (w3) at (10.61,5.8) [red] {};
\node (w4) at (7.6,5.8) [bk] {};
\draw[thick] (v1) -- (v2) -- (v4) -- (v3) -- (v1) -- (w1) -- (w2) -- (w3) -- (w4) -- (w1);
\draw[thick] (v2)  -- (w2);
\draw[thick] (v4)  -- (w3);
\draw[thick] (v3)  -- (w4);
\draw[thick] (x1)  -- (v1);
\draw[thick] (x2)  -- (v2);
\draw[thick] (x3)  -- (v4);
\draw[thick] (x4)  -- (v3);
\draw[thick] (y1)  -- (w1);
\draw[thick] (y2)  -- (w2);
\draw[thick] (y3)  -- (w3);
\draw[thick] (y4)  -- (w4);
\end{tikzpicture}
     \caption{The design $H_2'$ on $Q_4$.  }
     \label{fig:hamming lift}
 \end{figure}

\begin{theorem}\label{Lift Hamm}
\emph{The lifted Hamming code $H_r'$ is a $2^r$-design on $Q_{2^r}$, and $H_r''$ is a $(2^r+1)$-design on $Q_{2^r+1}$. We have that $$\eff (H_r') = \frac{2^{2^r-r}}{2^{2^r}-\binom{2^r}{2^{r-1}}} \sim \frac{1}{2^r} \quad \text{and} \quad \eff (H_r'') = \frac{2^{2^r-r+1}}{2^{2^r+1}-\binom{2^r+1}{2^{r-1}}} \sim \frac{1}{2^r} .$$}
\end{theorem}
\begin{proof}
The row spans of $M_r' $ and $M_r''$ are contained in $\Lambda_{2^{r-1}}$. By Theorem \ref{lincode thm}, $H_r'$ and $H_r''$ then integrate all eigenspaces but $\Lambda_{2^{r-1}}$.  The frequency ordering of $Q_{2^r} $ is
$$\L_0 \equiv \L_{2^r}< \L_1\equiv \L_{2^r-1} < \ldots <\L_{2^{r-1}} ,$$
and the frequency ordering on the eigenspaces of $Q_{2^r+1} $ is 
$$\L_0 \equiv \L_{2^r+1}< \L_1 \equiv \L_{2^r} <\ldots < \L_{2^{r-1}} \equiv \L_{2^{r-1}+1} ,$$
  Since $\Lambda_{2^{r-1}}$ is last by frequency for $Q_{2^r}$, $H_r'$ is a $2^r$-design on $Q_{2^r}$.  Since $\Lambda_{2^{r-1}}$ may be ordered last by frequency for $Q_{2^r+1}$, $H_r''$ is a $(2^r+1)$-design on $Q_{2^r+1}$. We calculate the efficacy of these designs both exactly and asymptotically in the same manner as in Theorem \ref{Main Hamm}.
\end{proof}

We show in Theorem \ref{lifted optimal} that $H_2'$ is optimal on $Q_{4}$, but it is otherwise unknown whether lifted Hamming codes are optimal. Geometrically, the lift $H_r'$ to the $2^r$- cube is an embedding of two copies of $H_r$, one on the $x_{2^r}= 0$ facet and one on the $x_{2^r}= 1$ facet. Likewise, $H_r''$ embeds four copies of $H_r$ in the $(2^r+1)$-cube.  Unfortunately, we cannot push this pattern further without a loss of efficacy, as $\Lambda_{2^{r-1}} $  cannot be ordered last by frequency in other dimensions.  
\begin{conjecture}
\emph{The lifted Hamming codes $H_r'$ and $H_r ''$ are optimal designs on their respective graphs $Q_{2^{r}}$ and $Q_{2^{r}+1}$.}
\end{conjecture}

To extend $H_r$ as a design to cubes of lower dimension, we consider projections. Consider the check matrix $\pi(M_r)$, the matrix with all but the last column of $M_r.$  For instance, for $r=3$, this provides a code on the 6-cube with check matrix
\[\pi(M_3)= \begin{bmatrix}
1 & 0 & 1 & 0 & 1 & 0 \\
0 & 1 & 1 & 0 & 0 & 1 \\
0 & 0 & 0 & 1 & 1 & 1 
\end{bmatrix}.\]
We denote the linear code with this check matrix by $\pi(H_r)$, which we call the projected Hamming code.
\begin{theorem} \label{Hamm Proj}
 \emph{The linear code $\pi(H_r)$ is a $(2^r-3)$-design on $Q_{2^r-2}$ with }
 $$\eff(\pi(H_r))= \frac{2^{2^r-r-2}}{2^{2^r-2} - \binom{2^r-2}{2^{r-1}-1}-\binom{2^r-2}{2^{r-1}}} \sim \frac{1}{2^r}.$$  
 \end{theorem} 
\begin{proof}
Let $\{a_1,\ldots, a_r\}$ be the rows of $M_r$, and let $$a_I = \sum_{i\in I}a_i \text{ for some nonempty }I \subseteq [r].$$ We recall that $a_I$ is an element of the simplex code, and hence has weight $2^{r-1}$. Since the last coordinate of $a_i$ is 1 for each $i=1,\ldots,r$, the last coordinate of $a_I$ is $|I| \mod 2 $.  Thus  $\pi(a_I)$ has weight $2^{r-1}$ when $|I|$ is even and weight $2^{r-1}-1$ when $|I|$ is odd. 
Therefore $\pi(H_r)$ integrates all eigenspaces except for $\L_{2^{r-1}-1}$ and $\L_{2^{r-1}}$.  The frequency ordering on $Q_{2^r-2}$ is
$$\L_0 \equiv \L_{2^r-2}< \L_1 \equiv \L_{2^r-3} < \ldots<\L_{2^{r-1}-2} \equiv \L_{2^{r-1}} <\L_{2^{n-1}-1} .$$
Thus the eigenspaces $\L_{2^{r-1}-1}$ and $\L_{2^{r-1}}$ may be ordered last by frequency.  The efficacy calculation follows since  $\dim \L_{2^{r-1}-1} =  \binom{2^r-2}{2^{r-1}-1}$ and $\dim \L_{2^{r-1}} =  \binom{2^r-2}{2^{r-1}}$. The asymptotics follow similarly to Theorem \ref{Main Hamm}.
\end{proof}

While we conjectured that $H_r$ and its lifts are optimal, we are less optimistic about $\pi(H_r)$. On $Q_6$, we have that $\eff(\pi(H_3))=8/29 \approx .276$. Lemma \ref{simple cube designs} provides a design $W$ on $Q_6$ with $\eff(W) = 4/14 \approx .286.$ As $r$ grows, we suspect there may be designs which are more effective than $\pi(H_r)$ on $Q_{2^r-2}$.

\end{subsection}

\begin{section}{Extremal Designs}
We next turn to the concept of extremal designs, as introduced in \cite{Golubev}. 
\begin{definition}[\cite{Golubev}] \label{extremal def}
  \emph{An \emph{extremal design} on a regular graph $G=(V,E)$ is a subset $W\subset V$ which integrates all but one eigenvector \textcolor{cyan}{in some eigenbasis of $L$.} }
\end{definition}
If $G=(V,E)$ is regular, $W\subset V$ is extremal if $\ones_W$ can be written as a linear combination of $\ones$ and one other eigenvector of $L$ by Lemma \ref{proof technique}. Moreover, this condition is equivalent to being extremal by the following lemma.
 
 \begin{lemma} \emph{Let $G=(V,E)$ be regular. A proper subset $W\subset V$ cannot integrate every eigenvector of $L$.}
\end{lemma}
\begin{proof}
Let $G = (V,E)$ be a regular graph on $n$ vertices and let $\{ \ones= \phi_1, \phi_2, \ldots, \phi_n\}$ be an orthonormal basis of eigenvectors for $L$. Suppose a subset $W\subseteq V$ integrates every eigenvector in this basis, that is, $ \phi_j^T \ones_W = 0 $ for each $j=2,\ldots, n.$ We can expand $\ones_W$ in this basis: $ \ones_W = \a_1 \ones + \sum_{i=2}^n \a_i\phi_i$ for some $\a_i\in\RR$. Therefore,
$$0 = \phi_j^T \ones_W = \phi_j^T \left(  \a_1 \ones + \sum_{i=2}^n \a_i\phi_i \right) = \a_1 \phi_j^T \ones +\a_j = \a_j$$
 for each $ j = 2,\ldots, n.$ Thus $\ones_W = \a_1 \ones $ implies $\a_1 =1$ and $W= V$. 
\end{proof}

 Golubev's two main results show that subsets which meet known spectral bounds are extremal. We start with a quick positive result. 
 
\begin{corollary}
\emph{The Hamming code $H_r$ and its lifts $H_r'$ and $H_r''$ are extremal designs in $Q_{2^r-1}, Q_{2^r}$, and $Q_{2^r+1}$, respectively.} \label{extremal hamm}
\end{corollary}
\begin{proof}
These designs integrate all but one eigenspace by Theorems \ref{Main Hamm} and \ref{Lift Hamm}. By Lemma \ref{change of basis}, there is an eigenbasis for the unintegrable eigenspace such that a design fails to integrate only one eigenvector of the eigenbasis.
\end{proof}

\textcolor{cyan}{This principal applies more generally.  A graphical design is extremal as in Definition \ref{extremal def} if and only if it integrates all but one eigenspace.}

 \subsection{Graphical Designs from the Hoffman Bound}
 
Recall that the \emph{stability number} of a graph is $\alpha(G) = \{\max |W|: W\subset V \text{ is a stable set}\}.$
Golubev's first result makes use of the Hoffman bound. Though often provided as the reference, Hoffman's bound does not appear in \cite{Hoffman}. The murky origins of the Hoffman bound are described in \cite{haemer}.
\begin{theorem}(Hoffman)  \label{Hoffman}
\emph{Let $G$ be a regular graph on $n$ vertices, and let $\l_{\min}$ be the least eigenvalue of $L$. 
Then,
$$\frac{\a(G)}{n} \leq \frac{-\l_{\min}-1}{-\l_{\min}}.$$}
\end{theorem}
\noindent
Golubev shows that stable sets which attain this bound are extremal designs.

\begin{theorem}(\cite[Theorem 2.2]{Golubev})
\emph{Let $G$ be a regular graph on $n$ vertices for which the Hoffman bound is sharp.  Let
$W \subset V$ be a stable set realizing $\alpha(G),$ i.e.
$$ \frac{|W|}{n} = \frac{\alpha(G)}{n}= \frac{-\l_{\min}-1}{-\l_{\min}}.$$ Then $W $ is an extremal design.}\label{GolHoff}
\end{theorem}

As we will now show, extremal designs found through Golubev's methods are unlikely to perform well in the frequency ordering.  
\begin{theorem}
\emph{A stable set for which the Hoffman bound is sharp is unable to integrate the eigenspace corresponding to $\l_{\min}$.}  \label{myHoff}
\end{theorem}

\begin{proof}
Let $G=(V,E)$ be a regular graph for which the Hoffman bound is sharp, and let $W \subset V$ be a stable set which achieves the Hoffman bound.  Let $\ones_W = \sum_{i=1}^n \a_i\phi_i$, $\a_i \in \RR$, $\phi_1 = \ones$ be a decomposition of $\ones_W$ with respect to a fixed eigenbasis of $L$. In the proof of the Hoffman bound (see \cite[Theorem 2.1]{Golubev}), there is a chain of inequalities 
\begin{align*}
0 &= \langle (L+I)\ones_W, \ones_W \rangle = \sum_{i=1}^n (\l_i+1)\a_i^2 \\
&\leq \a_1^2 + (\l_{\min} +1)\sum_{i=2}^n \a_i^2 = -\l_{\min}\frac{|W|^2}{n}+ (\l_{\min}+1)|W|.    
\end{align*} 
In order for the bound to be sharp, the proof of \cite[Theorem 2.2]{Golubev} notes that this chain of inequalities must be sharp.  If $$\sum_{i=1}^n (\l_i+1)\a_i^2 = \a_1^2 +  (\l_{\min}+1)\sum_{i=2}^n \a_i^2,$$
then for $i \geq 2$, $\a_i=0$ whenever $\l_i \neq  \l_{\min}$. 
Hence we can write $\ones_W = \ones + \b \phi$, where $L\phi =\l_{\min} \phi$ , $\b\in \RR$. Thus $W$ cannot integrate the eigenspace for $ \l_{\min}.$
 \end{proof}
\textcolor{cyan}{Without further information, one might expect the minimum eigenvalue of $L$ to be near the lower bound of $-2$, which we recall is early in the frequency ordering. } This conundrum is seen most strikingly in bipartite graphs.

\begin{example}
If $G = (U \sqcup V, E)$ is bipartite, then the eigenvalue $\l_{\min} = -2$ of $L$ is second in the frequency ordering. Thus an extremal design which achieves the Hoffman bound on a bipartite graph is at best a 1-graphical design.
  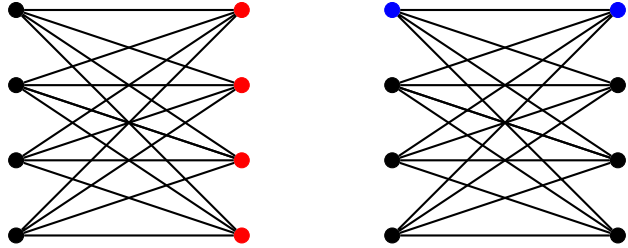
\begin{figure}[h!]
     \centering
\begin{tikzpicture}
\tikzstyle{bk}=[fill,circle, draw=black, inner sep=2 pt]
\tikzstyle{red}=[fill =red,circle, draw=red, inner sep=2 pt]
\tikzstyle{bl}=[fill=blue,circle, draw=blue, inner sep=2 pt]
%% hoffman subset 
\node (v1) at (0,0) [bk] {};
\node (v2) at (0,1) [bk] {};
\node (v3) at (0,2) [bk] {};
\node (v4) at (0,3) [bk] {};
\node (w1) at (3,0) [red] {};
\node (w2) at (3,1) [red] {};
\node (w3) at (3,2) [red] {};
\node (w4) at (3,3) [red] {};
\draw[thick] (v1) -- (w1) -- (v2) -- (w2) -- (v3) -- (w3) -- (v4) -- (w4);
\draw[thick] (v1) -- (w2) -- (v4) -- (w1);
\draw[thick] (w3) -- (v2) -- (w4) -- (v3) -- (w2);
\draw[thick] (v3) -- (w1);
\draw[thick] (w3) -- (v1) --(w4);
%% best subset
\node (v1) at (5,0) [bk] {};
\node (v2) at (5,1) [bk] {};
\node (v3) at (5,2) [bk] {};
\node (v4) at (5,3) [bl] {};
\node (w1) at (8,0) [bk] {};
\node (w2) at (8,1) [bk] {};
\node (w3) at (8,2) [bk] {};
\node (w4) at (8,3) [bl] {};
\draw[thick] (v1) -- (w1) -- (v2) -- (w2) -- (v3) -- (w3) -- (v4) -- (w4);
\draw[thick] (v1) -- (w2) -- (v4) -- (w1);
\draw[thick] (w3) -- (v2) -- (w4) -- (v3) -- (w2);
\draw[thick] (v3) -- (w1);
\draw[thick] (w3) -- (v1) --(w4);
\end{tikzpicture}
     \caption{The complete bipartite graph $K_{4,4}$. The red subset attains the Hoffman bound and so is extremal. However, it is only a 1-graphical design. The blue subset is a 2-graphical design, extremal, and optimal. }
     \label{fig:bipartite}
 \end{figure}
\end{example}

We can make use of Theorem \ref{GolHoff} to find optimal designs on Cayley graphs. Recall that a connected, undirected Cayley graph $\Gamma(H,S)$ arises from a group $H$ and symmetric generating set $S\subseteq H$, where symmetric means $S = S^{-1}:=\{s^{-1}:s\in S\}$. The vertex set of $\Gamma(H,S)$ is $H$, and $xy$ is an edge if $y = xs$ for some $s\in S$. If $H$ is an abelian group with $n$ elements, then there are $n$ group characters $\chi:H \to \CC ^*$, each of which provides an eigenvector $(\chi(h))_{h\in H}$ of $L$ with eigenvalue $(1/|S|)\sum_{s\in S} \chi(s) -1$.
Thus if $\chi$ is an eigenvector of $\Gamma(H,S)$ with eigenvalue $(1/|S|)\sum_{s\in S} \chi(s) -1 $, then $\chi$ is an eigenvector of $\Gamma(H,S')$ with eigenvalue $(1/|S'|)\sum_{s\in S'} \chi(s) -1$. Therefore an extremal design on a Cayley graph $\Gamma(H,S)$ is also extremal on  $\Gamma(H,S')$ for all other symmetric generating sets $S'$.
Theorem \ref{GolHoff} (see also Theorem \ref{GolCheeg}) gives a sufficient condition for a design to be extremal, though we have shown the frequency order of the unintegrated eigenspace may not be ideal. However,  given a different generating set, the unintegrated eigenspace may be last in the frequency ordering.

We use this idea to show that $H_2$ is optimal on $Q_3$. This is equivalent to using a sledgehammer on a thumbtack, but we think the strategy may be more generally useful.  Consider the distance graph $Q_d(N)$, with the same vertex set as $Q_d$, and edge set $E  = \{vw: 0 < d_H(v,w) \leq N\}$. Note that $Q_d(1) = Q_d$, and that $Q_d(N)$ is the Cayley graph $ \Gamma(H,S)$ with $H= \{0,1\}^d$ and $S=\{ \sum_{i\in I} e_i\}_{I\subseteq [d],|I|\leq N}$.
\begin{theorem}
 \label{H2 optimal}
The Hamming code $H_2 = \{000,111\}$ is an optimal design on $Q_3$.
\end{theorem}

\begin{proof}
We compute the spectral information of $L$ for $Q_3$ and $Q_3(2)$, described in Table \ref{tab:dist cube spectra}. Thus the Hoffman bound gives us that
$$ \frac{ \alpha(Q_3(2))}{ |V|}  \leq \frac{ - 1/3}{-4/3} = \frac{1}{4}.$$ The Hamming code $H_2 $ is a stable set in $Q_3(2)$ which attains the Hoffman bound, hence $H_2$ is extremal on $Q_3(2)$. By Theorem \ref{myHoff}, the eigenspace $H_2$ fails to integrate is $\L_2$, since $-4/3$ is the least eigenvalue of $L$ of $Q_3(2)$.  The eigenspace $\L_2$ can be ordered last by frequency on $Q_3$. Therefore, $\eff(H_2) = 2/5$. Since no single vertex integrates $\L_3$ by Lemma \ref{cube lambda n}, $H_2$ is then optimal on $Q_3$.
\end{proof}

\begin{table}[h!]
    \centering
    \begin{tabular}{c| c|c c}
    Eigenspace & $\dim(\L_i)$       &  $Q_3$ & $Q_3(2)$ \\ \hline
    $\L_0= \spanset\{\ones\} $      & 1 & \phantom{$-$}0 & \phantom{$-$}0 \\
    $\L_1= \spanset\{\phi_v: wt(v) =1 \} $ & 3&  $- 2/3$ & $-1$ \\
    $\L_2 = \spanset\{\phi_v: wt(v) =2 \}$& 3 & $-4/3$ & $-4/3$ \\
    $\L_3 = \spanset\{\phi_\ones \}$& 1 & $-2$ & $-1$
    \end{tabular}
    \caption{The spectra of $L$ for $Q_3$ and $Q_3(2)$.}
    \label{tab:dist cube spectra}
\end{table}

\subsection{Graphical Designs from the Cheeger Bound}

The other main result of \cite{Golubev} relies on the following variant of the Cheeger bound. \textcolor{cyan}{We use $E(W, V \setminus W)$ to denote the set of edges with one vertex in $W$ and the other in $V\setminus W$.}

\begin{theorem}(\cite{AM,Tanner})
\emph{Let $G$ be a connected $\d$-regular graph and $\l^*$ be the second largest eigenvalue of $L$. Then} 
$$ \underset{\vn \neq W \subset V}{\min}\frac{|V| |E(W, V\setminus W)|}{\d|W||V\setminus W|} \geq -\l^* .$$ 
\end{theorem}
\noindent
In a manner similar to Theorem \ref{GolHoff}, Golubev shows that subsets which attain the Cheeger bound are extremal.

\begin{theorem}(\cite[Theorem 2.4]{Golubev})
\emph{Let $G = (V,E)$ be a $\d$-regular graph for which the Cheeger bound is sharp. Suppose
$\vn \neq W \subset V$ realizes the Cheeger bound:}
$$ \frac{|V| |E(W, V\setminus W)|}{\d|W||V\setminus W|} = -\l^*.$$ \emph{Then $W $ is an extremal design. }\label{GolCheeg}
\end{theorem}
\noindent
The proof of this result in \cite{Golubev} indicates which eigenspace cannot be integrated.

\begin{theorem}
\emph{A subset $W\subset V$ for which the Cheeger bound is sharp cannot integrate the eigenspace with eigenvalue $\l^*$. } \label{myCheeg}
\end{theorem}

\begin{proof}
Let $G=(V,E)$ be a $\d$-regular graph, and suppose $W \subset V$ achieves the Cheeger bound. Let $\ones_W = \sum_{i=1}^n \a_i\phi_i$. $\a_i \in \RR$, $\phi_1 = \ones$ be a decomposition of $\ones_W$ with respect to a fixed eigenbasis of $L$. In the proof of the Cheeger bound (see \cite[Theorem 2.3]{Golubev}), there is a chain of inequalities which we abbreviate here:
\begin{align*}
    &\frac{1}{\d}|E(W, V\setminus W)| = \ldots 
    = \a_1( \sqrt{n}-\a_1) + \sum_{i=2}^n (\l_i+1)(-\a_i^2) \\
    &\geq \a_1( \sqrt{n}-\a_1) + (\l^*+1)\sum_{i=2}^n (-\a_i^2) = \ldots = \frac{|W|(n-|W|)}{n}(-\l^*).
    \end{align*}
In order for the bound to be sharp, the proof of \cite[Theorem 2.4]{Golubev} notes that this chain of inequalities must be sharp. Thus if $$ \sum_{i=2}^n (\l_i+1)\a_i^2
  = (\l^*+1) \sum_{i=2}^n \a_i^2, $$
  then for $i \geq 2$, $\a_i =0$ whenever $\l_i \neq \l^*$. Hence we can write $\ones_W = \ones + \b\phi$,  where $\phi$ is an eigenvector of $L$ with eigenvalue $\l^*$ and $\b \in \RR$. Therefore $W$ cannot integrate the eigenspace with eigenvalue $\l^*$.
\end{proof}

 Since $\l^*$ is the second largest eigenvalue of $L$, one might expect it to be near 0, which corresponds to an eigenspace early in the frequency ordering of $L$. Theorems \ref{myHoff} and \ref{myCheeg} may then be a way to tie expander graphs into the theory of graphical designs. If $\l_{\min}$ is near $-1$, then the Hoffman bound may work well with the frequency ordering. Similarly, if $\l^*$ is near $-1$, then the Cheeger bound may work well with the frequency ordering.
 
 \subsection{More on Cube Graphs}

We first compare several designs on $Q_d$.
   
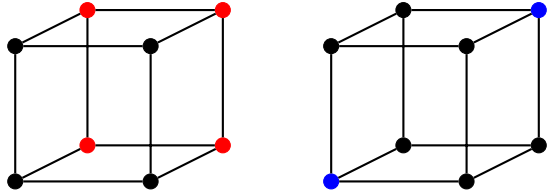
\begin{figure}[h!]
     \centering
\begin{tikzpicture}[scale = .6]
\tikzstyle{bk}=[fill,circle, draw=black, inner sep=2 pt]
\tikzstyle{red}=[fill =red,circle, draw=red, inner sep=2 pt]
\tikzstyle{bl}=[fill=blue,circle, draw=blue, inner sep=2 pt]
%% cheeger subset 
\node (v1) at (0,0) [bk] {};
\node (v2) at (3,0) [bk] {};
\node (v4) at (3,3) [bk] {};
\node (v3) at (0,3) [bk] {};
\node (w1) at (1.6,.8) [red] {};
\node (w2) at (4.6,.8) [red] {};
\node (w3) at (4.6,3.8) [red] {};
\node (w4) at (1.6,3.8) [red] {};
\draw[thick] (v1) -- (v2) -- (v4) -- (v3) -- (v1) -- (w1) -- (w2) -- (w3) -- (w4) -- (w1);
\draw[thick] (v2)  -- (w2);
\draw[thick] (v4)  -- (w3);
\draw[thick] (v3)  -- (w4);
%% best subset
\node (v1) at (7,0) [bl] {};
\node (v2) at (10,0) [bk] {};
\node (v3) at (7,3) [bk] {};
\node (v4) at (10,3) [bk] {};
\node (w1) at (8.6,.8) [bk] {};
\node (w2) at (11.6,.8) [bk] {};
\node (w3) at (11.6,3.8) [bl] {};
\node (w4) at (8.6,3.8) [bk] {};
\draw[thick] (v1) -- (v2) -- (v4) -- (v3) -- (v1) -- (w1) -- (w2) -- (w3) -- (w4) -- (w1);
\draw[thick] (v2)  -- (w2);
\draw[thick] (v4)  -- (w3);
\draw[thick] (v3)  -- (w4);
\end{tikzpicture}
     \caption{$Q_3$. The red subset, which attains the Cheeger bound, and the blue subset, $H_2$, both integrate the eigenspaces $\L_0, \L_1,$ and $\L_3$, which come first in the frequency ordering.  Since the Hamming code has fewer vertices, it is more effective. }
     \label{fig:cheeg v hamming}
 \end{figure}
Our Theorem \ref{lincode thm} can be thought of as an extension of \cite[Theorem 3.3]{Golubev}, which considers linear codes on $Q_d$ with $1 \times d$ check matrices. The most effective design on $Q_d$ by this method has efficacy $(2^{d-1})/(2^d-\binom{d}{\lceil d/2\rceil})>.5$. We recall the efficacies of $H_r$ and its lifts approach $ 1/2^r$. For a concrete example, consider the 7-cube. The most effective design from \cite[Theorem 3.3]{Golubev} will have efficacy  $64 /93 \approx .688$. We calculate $\eff(H_3) = 16/93 \approx .172$.
    
Section 3.5 of \cite{Golubev} shows that the Cheeger bound is attained for $Q_d$ by the subset $S =\{ v\in \{0,1\}^d: v_1 =1\}.$  In this graph, the eigenspace with eigenvalue $\l^* =  - 2/d$ of $L$ can at best be ordered fourth by frequency, behind $\l_1 =0, \l_d =-2$ and $\l_{d-1} = -2 + 2/d$.  Thus $S$ is a  $3$-design. Since $|S| = |V|/2$,  $\eff(S)=(2^{d-1})/(d+2)$, which grows exponentially in $d$.  \textcolor{cyan}{We also note that $Q_d$ is bipartite, where the bipartition separates even and odd weight vertices. Example 5.7 thus shows that the set of even (or odd) weight vertices is an extremal design on $Q_d$ which consist of $2^{d-1}$ vertices. }

Lastly, we show that optimal designs are a distinct concept from stable sets using $Q_4$ and the lifted Hamming code $H_2'$ as an example. Theorem \ref{GolHoff} is a link between graphical designs and maximum stable sets. The Hamming code $H_r$ is a maximum stable set in the graph $Q_{2^r-1}(2)$. In several of the striking examples shown in \cite{graphdesigns}, the graphical designs found are maximum stable sets. The proof of Theorem \ref{H2 optimal} was inspired by the work utilizing maximum stable sets in distance graphs following Delsarte's linear programming bound for codes (\cite{DelsarteThesis}).  This body of work computes upper bounds for the largest codes of a fixed distance in some setting through semidefinite programming. See \cite{VallentinSDPECC} for a more complete overview of this area. We find it natural to then wonder whether graphical designs are the same as maximum stable sets, possibly in a distance graph, or whether optimal designs are stable sets at all. The lifted Hamming code $H_r'$  as described in Section 4.4 (see Figure \ref{fig:hamming lift}) is not a stable set in $Q_{2^r}$, but $\eff(H_r')$ is small. Moreover, we have the following optimality result for $Q_4$.
\begin{theorem} \label{lifted optimal}
The lifted Hamming code $H_2'= \{ 0000,0001,1111,1110\}$ is an optimal design on $Q_4$, and no optimal design on $Q_4$ is a stable set.
\end{theorem}
\begin{proof}
Since $Q_4$ has an even number of vertices, Lemmas \ref{cube lambda n} and \ref{cube odd eigspaces} imply that if $W\subset Q_4$ is such that $|W| \leq 3$, then $W$ integrates at most the first two eigenspaces $\L_0$ and $\L_4$, a total of two eigenvectors.  Theorem \ref{Lift Hamm} shows that $H'_2$ integrates all but the last eigenspace in the frequency ordering.  Thus the minimal efficacy of a design on $Q_4$ is $4/10$, and $H_2'$ achieves this minimum. A brute force search finds 16 4-element subsets in $Q_4$ which are optimal, none of which are stable sets.  
\end{proof}
%%Are all 16 of these designs somehow equivalent to $H_2'$? If so, how?
\end{section}

\begin{section}{Association Schemes and $t$-designs}
At first glance, there are obvious similarities between  $t$-designs on association schemes, introduced in full generality by Delsarte \cite{DelsarteThesis}, and our graphical designs.  We will show that they are not the same concepts despite these formal similarities. We will exhibit extremal and optimal graphical designs on association schemes that are not $t$-designs (Proposition \ref{johnson graph counterex} and Example \ref{optimal but not t johnson}) and also show that that even if a graphical design is a $t$-design, it may be a $k$-graphical design where $k$ and $t$ are quite different (Proposition \ref{hamm better GD than TD}).
We start with a brief introduction to association schemes. This exposition is based primarily on \cite[Chapter 21]{ECC}.

\subsection{Association Schemes}

\begin{definition}[see Chapter 21 of \cite{ECC}]
 \emph{A (symmetric) association scheme $(X, \cR)$ with $s$ classes is a finite set $X$ and $s+1$ relations $\cR = \{R_0,\ldots R_s\}$ on $X$ such that 
\begin{enumerate}
    \item $(x,y) \in R_i \iff (y,x) \in R_i$.
    \item For all $x,y\in X$, $(x,y) \in R_i$ for exactly one $i$.
    \item $R_0 = \{(x,x): x\in X\}$.
    \item If $(x,y) \in R_q$, then the number of $z\in X$ such that $(x,z)\in R_i$ and $(y,z) \in R_j$ is a constant $\a_{ijq}$ which does not depend on the choice of $x$ or $y$.
    \end{enumerate}}
\end{definition}
If $(x,y) \in R_i$, we say that $y$ is an  \emph{$i$-th associate} of $x$. It can help to visualize an association scheme as a complete graph on $X$ with labeled edges as in Figure \ref{fig: hamm scheme}.  

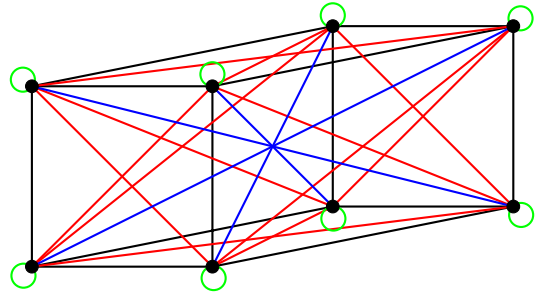
\begin{figure}[h!]
\begin{center}
\begin{tikzpicture}[scale = .8]

%%% edges i =1
\draw[thick] (0,0) -- (3,0) -- (3,3) -- (0,3) -- (0,0) -- (5,1) -- (8,1) -- (8,4) -- (5,4) -- (5,1);
\draw[thick] (3,0)  -- (8,1);
\draw[thick] (3,3)  -- (8,4);
\draw[thick] (0,3)  -- (5,4);
%%% edges i=2
\draw[thick,red] (3,0)  -- (8,4);
\draw[thick,red] (3,3)  -- (8,1);
\draw[thick,red] (0,0)  -- (5,4);
\draw[thick,red] (0,3)  -- (5,1);
\draw[thick,red] (0,0)  -- (3,3);
\draw[thick,red] (0,3)  -- (3,0);
\draw[thick,red] (8,1)  -- (5,4);
\draw[thick,red] (8,4)  -- (5,1);
\draw[thick,red] (0,0)  -- (8,1);
\draw[thick,red] (0,3)  -- (8,4);
\draw[thick,red] (3,0)  -- (5,1);
\draw[thick,red] (3,3)  -- (5,4);
%%% edges i=3
\draw[thick,blue] (0,0)  -- (8,4);
\draw[thick,blue] (0,3)  -- (8,1);
\draw[thick,blue] (3,0)  -- (5,4);
\draw[thick,blue] (3,3)  -- (5,1);
%%% edges i=0
\draw[thick,green] (-.14,-.15 ) circle (.2);
\draw[thick,green] (3.02,-.19 ) circle (.2);
\draw[thick,green] (-.15,3.11 ) circle (.2);
\draw[thick,green] (3,3.2 ) circle (.2);
\draw[thick,green] (5.01,.8 ) circle (.2);
\draw[thick,green] (8.13,.86) circle (.2);
\draw[thick,green] (5,4.18 ) circle (.2);
\draw[thick,green] (8.12,4.13 ) circle (.2);
%% vertices
\draw[fill=black] (0,0) circle (3pt);
\draw[fill=black] (3,0) circle (3pt);
\draw[fill=black] (0,3) circle (3pt);
\draw[fill=black] (3,3) circle (3pt);
\draw[fill=black] (5,1) circle (3pt);
\draw[fill=black] (8,1) circle (3pt);
\draw[fill=black] (5,4) circle (3pt);
\draw[fill=black] (8,4) circle (3pt);
\end{tikzpicture}
              \caption{ The Hamming association scheme on the 3-cube, where $X = \{0,1\}^3 $ and $R_i = \{(x,y): d_H(x,y) =i\}$. We indicate $R_0$ by green, $R_1$ by black, $R_2$ by red, and $R_3$ by blue. }
              \label{fig: hamm scheme}
      \end{center} 
      \end{figure}
      
      \noindent
      We can reformulate Definition 6.1 using matrices.  Let $A_i \in \RR^{X \times X}$ be the adjacency matrix of $R_i$.  That is,
      \vspace{-.1 in}
\[ (A_i)_{xy} =
\begin{cases}
1 & (x,y) \in R_i \\
0 & (x,y) \notin R_i
\end{cases}.
\]
Then conditions (1) through (4) of Definition 6.1 are
\begin{enumerate}
    \item $A_i$ is a symmetric matrix for each $i.$
    \item $\sum_{i=0}^{s} A_i$ is the all-ones matrix.
    \item $A_0 =I$.
    \item $A_iA_j = \sum_{q=0}^s \a_{ijq}A_q = A_jA_i$.
    \end{enumerate}

Consider the real vector space $\cA = \{\sum_{i=0}^s \b_i A_i : \b_i \in \RR\}$.  By (2) of Definition 6.1, the dimension of $\cA$ is $s+1$, and by (4) of Definition 6.1, $\cA$ is closed under matrix multiplication, which is commutative.  Thus $\cA$ is an associative, commutative algebra, called the \textit{Bose-Mesner algebra} of the association scheme after \cite{BMalgebras}. It can be shown that $\cA$ is semisimple, and thus has a unique basis of primitive idempotents $J_0, \ldots, J_s.$  These $J_i \in \RR^{X \times X}$ satisfy
\begin{enumerate}
    \item $J_i$ is a symmetric matrix for each $i.$
    \item $J_i^2 = J_i, \ i=0,\ldots, s.$
    \item $J_iJ_j = 0 $ for $ i \neq j$.
    \item $\sum_{i=0}^s J_i = I$. 
\end{enumerate}
We always take $J_0$ to be the all-ones matrix scaled appropriately, but in general, there is no ordering imposed on the $J_i$. We now have two distinct bases for $\cA$. Let's express one in terms of the other:
$$A_j = \sum_{i=0}^{s} p_j(i) J_i$$
for some $p_j(i) \in \RR$.  By the properties of idempotents, we then have that $$A_j J_i = p_j(i) J_i.$$ \textcolor{cyan}{Thus the eigenspace of $A_j$ with eigenvalue $p_j(i)$ contains $\col(J_i),$ the column span of $J_i$. Each eigenspace of $A_j$ consists of the collected column spans of some of the matrices $J_i$. So if $p_j(i) \neq p_j(q)$ for each $q\neq i$, then the eigenspace with eigenvalue $p_j(i)$ is exactly $\col(J_i)$}. To be clear, distinct adjacency matrices $A_j $ and $A_q$ will generally have different sets of eigenvalues $\{p_j(i)\}_{i=0}^s$ and $\{p_q(i)\}_{i=0}^s$, respectively, \textcolor{cyan}{though it can be that $p_j(i) = p_q(i)$.  However, for each $i$, $\col(J_i)$ is contained in the eigenspaces of $p_j(i)$ and $p_q(i)$.} The next lemma mirrors the discussion of Cayley graphs on a fixed group with different generating sets preceding Theorem \ref{H2 optimal}. 

\begin{lemma}
\emph{Let $(X, \cR)$ be an $s$-class association scheme,
let $I \subset [s]$, and denote $R_I = \cup_{j\in I} R_j$. \textcolor{cyan}{If $Y\subset X$ integrates $\col(J_i)$ for all but one $i\in [s]$, then $Y$ is an extremal design in the graph $G_I = (X, R_I)$. The eigenspace of $A_I = \sum_{j\in I} A_j$, the adjacency matrix of $G_I$, that $Y$ does not integrate has eigenvalue $\sum_{j\in I} p_j(i)$ \label{distgraphs}}}
\end{lemma}

\begin{proof}
\textcolor{cyan}{ Let $\phi \in \col( J_i)$, $i\in[s].$  Since $A_j \phi = p_j(i)\phi$ for each $j =0,\ldots s,$ we have that 
\begin{align*}
  A_I \phi =  \sum_{j\in I} A_j \phi =  \sum_{j\in I} p_j(i)\phi =  \left(\sum_{j\in I} p_j(i) \right)\phi.
\end{align*}
So we see that $\phi$ is an eigenvector of $A_I$ with eigenvalue $\sum_{j\in I} p_j(i)$. }

\textcolor{cyan}{ Suppose $Y$ integrates all but $\col (J_{i'})$ for some $i'\in [s]$. Then the only eigenvectors $Y$ cannot integrate lie in $\col (J_{i'})$ which is contained in the eigenspace with eigenvalue $\sum_{j\in I} p_j(i')$ for $A_I$.}
\end{proof}

Note that the above lemma was proven in terms of the adjacency matrices $A_i$, not the Laplacian. 
It is sufficient to work with the adjacency matrices in the case of regular graphs by the following lemma.

\begin{lemma}
 \emph{Let $G = (V,E)$ be a $\d$-regular graph with adjacency matrix $A$.  Then $v$ is an eigenvector of $A$ with eigenvalue $\l$ if and only if $v$ is an eigenvector of $L = AD^{-1} - I$ with eigenvalue $\l/\d -1$. }
\end{lemma}
\begin{proof} If $G$ is $\d$-regular, then $AD^{-1} = A/\d$.  So,
\begin{align*}
A v = \l v \iff
     \frac{1}{\d}A v = \frac{\l}{\d} v \iff 
  (AD^{-1} - I) v = \left(\frac{\l}{\d}-1\right) v.
\end{align*}

\end{proof}

\subsection{$T$-designs, $t$-designs, and the Johnson scheme.}

We begin relating $T$-designs in association schemes to graphical designs in our sense. We then examine the more structured case of $t$-designs in cometric association schemes. We use the Johnson scheme as an example where the ordering of eigenspaces due to the cometric structure is incompatible with the frequency ordering.

\begin{definition}[see Theorem 3.10 of \cite{DelsarteThesis}]
 \emph{ Let $(X, \cR)$ be an association scheme with $s$ classes and $T \subseteq [s].$ A \emph{$T$-design} is  $Y \subset X$ such that $J_i \ones_Y =0$ for each $i\in T$. }\end{definition}
\textcolor{cyan}{From the definition, it is clear that $T$-designs integrate the eigenvectors in $\col(J_i)$ for $ i \in [T]$ in the sense of graphical designs.  However, without more information, there is no guarantee that any $\col(J_i)$ spans an eigenspace, and we know nothing about where any integrated eigenspaces are in the frequency ordering. }

 \begin{lemma} 
\label{TD neq GD}
 \emph{\textcolor{cyan}{Let $(X, \cR)$ be an association scheme with $s$ classes, $I \subset [s]$ be an index set, and let $G_I = (X,R_I)$ be the graph with edge set $R_I = \cup_{i\in I} R_i$.  If $Y\subset X$ is a $k$-graphical design, then $Y$ is a $T$-design for some $|T| \geq k-1$. }}
\end{lemma}

\begin{proof}
\textcolor{cyan}{Let $ \L_1 = \spanset\{\ones\} \leq \L_2 \leq \ldots \leq \L_m$ be the eigenspaces of $G_I$ ordered by frequency. Suppose $Y\subset X$ integrates $\L_1, \L_2, \ldots, \L_{k}$. For $j \in [k]$, $\L_j$ contains the column span of at least one idempotent, call it $J_{i_j}$. Let $T = \{i_j: j=2,\ldots, k\}$.  Since $J_{i_j}^T \ones_Y =0$ for each $j=2,\ldots,k,$ we have that $Y$ is a $T$-design.}
\end{proof}

If an association scheme has the property that it is \emph{cometric}, or equivalently \emph{$Q$-polynomial}, there is a natural ordering of the idempotents as $J_0, J_1,\ldots, J_s$. For definitions and more details on this matter, see Sections 2.7 and 2.8 of \cite{BCN} and Section 5.3 in \cite{DelsarteThesis}. An association scheme can have at most two idempotent orderings which are cometric, and it is fairly unusual to have more than one \cite{Suzuki}. If an association scheme is cometric, one can define the following.

\begin{definition}[See Section 5.3 in \cite{DelsarteThesis}]
 \emph{For a cometric association scheme $(X, \cR),$  $Y\subset X$ is called a $t$-design on $X$ if $Y$ is a $T$-design for $T = [t]$. }
\end{definition} 

Due to the formal similarities, it is natural to wonder whether graphical designs reduce to classical $t$-designs for graphs from cometric association schemes. To show that this is not the case, we 
introduce the Johnson scheme. Let $X$ be all $s$-element subsets of $[l]$, and call $A,B \in X$  $i$-th associates if $|A \cap B| = s-i$.  This forms a cometric association scheme with $s$ classes,  which we call the $(l,s)$ Johnson scheme.  \textcolor{cyan}{Typically, one considers  $s\in \left[ \lfloor l/2 \rfloor\right]$, as $s=0$ is the trivial graph, and the $(l,s)$ Johnson scheme is equivalent to the $(l,l-s)$ Johnson scheme by symmetry of the binomial coefficient}. See \cite[Section 4.2.1]{DelsarteThesis} for the eigenspace ordering. We can understand $t$-designs \textcolor{cyan}{on the $(l,s)$ Johnson scheme} as classical $t$-$(l,s,\gamma)$ designs.

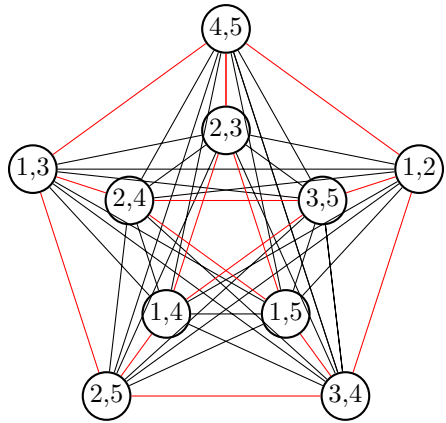
\begin{figure}[h!]
\begin{center}
    \begin{tikzpicture}[scale = .9]
        \tikzstyle{bk}=[circle, fill = white,inner sep= 1 pt,draw, thick]
            %%% R0
%%nodes
\node (v1) at (3* .951, 3*.309) [bk] {1,2};
\node (v2) at (0,3) [bk] {4,5};
\node (v3) at (3* -.951, 3*.309) [bk] {1,3};
\node (v4) at (3*-.588,3*-.809)  [bk] {2,5};
\node (v5) at (3*.588,3*-.809)  [bk] {3,4};
\node (w1) at (1.5* .951, 1.5*.309) [bk] {3,5};
\node (w2) at (0,1.5) [bk] {2,3};
\node (w3) at (1.5* -.951, 1.5*.309) [bk] {2,4};
\node (w4) at (1.5*-.588,1.5*-.809)  [bk] {1,4};
\node (w5) at (1.5*.588,1.5*-.809)  [bk] {1,5};
%%R1
\draw (v1) --(w2) -- (v3) -- (w4) -- (v5) --(v2) -- (w3) -- (v4) -- (w5) -- (v1) -- (w3) -- (v5) -- (v3) -- (w5) -- (w4) -- (w3) -- (w2) -- (w1) -- (w5) --(v2) ;
\draw (v1) -- (w4) -- (v2) -- (v4) -- (w1) -- (v5) -- (w2) -- (v4) -- (v1)  --(v3) -- (w1) --(v5) --(v2);
\draw (v2) -- (w1);
\draw[red] (w1) -- (w3) -- (w5) -- (v5) -- (v1) -- (w1) -- (w4) -- (v4) -- (v5);
\draw[red] (v4) -- (v3) -- (w3);
\draw[red] (v3) -- (v2) --(w2) -- (w5);
\draw[red] (v1) -- (v2) -- (w2) -- (w4);
%%R2
\end{tikzpicture}
    \caption{The $(5,2)$ Johnson scheme. We indicate $R_1$ by black and $R_2$ by red. We omit the loops at each vertex from $R_0$. $G_2$ is the Petersen graph, and $G_1$ is its complement.}
    \label{fig: J(5,2)}
        \end{center}
\end{figure}

\begin{definition}[see Section 2.5 of \cite{ECC}]
 \emph{ A $t$-$(l,s,\g)$ design is a collection of $s$-element subsets of $[l]$, called \textit{blocks}, such that any subset of $t$ elements chosen from $[l]$ is in contained in exactly $\g$ blocks.}\end{definition}

\begin{example}
The projective plane $PG(2,2)$, as visualized below, forms a $2$-$(7,3,1)$ design.  The nodes are elements of $[7]$, and the lines visualize blocks of size 3.  For any two points, there is exactly one line through them. The lines are elements of the $(7,3)$ Johnson scheme.
\end{example}

\begin{center}
\begin{tikzpicture}[scale = .8]
\tikzstyle{point}=[fill=white ,circle, draw=black, thick, inner sep=2 pt]
\node (v7) at (0,0) [point] {4};
\draw  circle (1 cm);
\node (v1) at (90:2cm) [point] {1};
\node (v2) at (210:2cm) [point] {5};
\node (v4) at (330:2cm) [point] {7};
\node (v3) at (150:1cm) [point] {2};
\node (v6) at (270:1cm) [point] {6};
\node (v5) at (30:1cm) [point] {3};
\draw (v1) -- (v3) -- (v2);
\draw (v2) -- (v6) -- (v4);
\draw (v4) -- (v5) -- (v1);
\draw (v3) -- (v7) -- (v4);
\draw (v5) -- (v7) -- (v2);
\draw (v6) -- (v7) -- (v1);
\end{tikzpicture}
\end{center}

Classical $t$-$(l,s,\g)$ designs are connected to the Johnson scheme by the following.

\begin{theorem}[Theorem 4.7 of \cite{DelsarteThesis}]
\label{Johnson tdesign}
 \emph{ Let $(X, \cR)$ be the $(l,s)$ Johnson scheme.  Then $Y \subset X$ is a  $t$-$(l,s,\g)$ design for some $\g$ if and only if $Y$ is a $t$-design on $X$ \textcolor{cyan}{for the cometric ordering given in \cite[Section 4.2.1]{DelsarteThesis}.}}
 \end{theorem}

 \textcolor{cyan}{We note that for $s >2$, the ordering of the idempotents in \cite[Section 4.2.1]{DelsarteThesis} is the only cometric ordering possible.}
  We will exhibit an optimal graphical design on a graph from a cometric association scheme which is not a classical $t$-design for any $ t>0.$ The Kneser graph $KG(l,s)$ is the graph $G_s=(X,R_s)$ from the $(l,s)$ Johnson scheme, which is to say its vertices are $s$-element subsets of $[l]$, and there is an edge between subsets which do not intersect.  

\begin{proposition} \label{johnson graph counterex}
 \emph{Let $(X, \cR)$ be the $(l,s)$ Johnson scheme with $s>2$, and let $Y \subset X$ be the $s$-element subsets of $[l]$ with exactly one fixed element in common. Then $Y$ is extremal on $G_s= KG(l,s)$, but $Y$ is not a $t$-design.}
\end{proposition}
\begin{proof}
The Hoffman bound (see Theorem \ref{Hoffman}) is sharp for $KG(l,s)$. Fix an element in $[l]$, and let $Y\subset X$ be the $s$-element subsets of $[l]$ which contain that fixed element.  Section 3.2 of \cite{Golubev} shows that $Y$ is a maximum stable set in $KG(l,s)$ and hence extremal on $KG(l,s) $ by Theorem \ref{GolHoff}.

 It suffices to show that $Y$ is not a $t$-$(l,s,\g)$ design for any $t$ by Theorem \ref{Johnson tdesign}. Let the shared element among the subsets of $Y$  be $l$. That is, $A \in Y$ is of the form  $$A = \{ l\} \cup  \binom{[l-1]}{s-1}.$$
Let $t \in [s]$. Consider the $t$-element subset $S = \{ l-(t-1), \ldots,l-1, l\}$. Then $S$ is contained in $\binom{l-t}{s-t}$
$s$-element subsets, all of which contain $l$ and so are in $Y$. Thus if $Y$ is a $t$-$(l,s,\g)$ design, it must be that $\g = \binom{l-t}{s-t}.$

We will exhibit a subset of $t$ elements which is not contained in $\g= \binom{l-t}{s-t}$ elements of $Y$. Consider $S' = \{ l-t, \ldots,l-1\}.$  Since $l \notin S'$, a subset in $Y$ containing $S'$ will be of the form 
$$ \{l\} \cup S' \cup  \binom{[l-(t+1)]}{s-(t+1)}.$$
Hence there are $ \binom{l-(t+1)}{s-(t+1)} \neq \g$ subsets in $Y$ containing $S'$.
Thus for any $t \in [s]$, there is no $\g$ such that $Y$ is a $t$-$(l,s,\g)$ design. 
\end{proof}

\begin{corollary}
Let $(X, \cR)$ be the $(l,s)$ Johnson scheme with $s>2$, let $I\subset[s]$ and let $G_I=(X, R_I)$.  The subset $Y$ in Proposition \ref{johnson graph counterex} is extremal on $G_I$.
\end{corollary}

\begin{proof}
\textcolor{cyan}{The Kneser graph $KG(l,s)$ has $s+1$ distinct eigenvalues given by $$(-1)^j \binom{l-s-j}{s-j}, \ j=0,1,\ldots, s.$$ By a counting argument, it must be that the distinct eigenspaces of $KG(l,s)$ are precisely $\col(J_i)$, $i =0,1,\ldots, s$. So if $Y$ is extremal on $KG(l,s)$, then $Y$ integrates $\col(J_i)$ for all but one $i\in[s]$.
Thus by Lemma \ref{distgraphs}, $Y$ is extremal on $G_I$.}
\end{proof}

\begin{figure}[h!]
\begin{center}
    \begin{tikzpicture}[scale =.8]
\tikzstyle{wnode}=[circle, fill = white,inner sep=1pt,draw, very thick]
\tikzstyle{rnode}=[circle, fill = white,inner sep=1pt,draw =red,  thick]

\node (v2) at (120:3) [rnode] {1,2};
\node (v3) at (336:3) [rnode] {1,3};
\node (v4) at (168:3) [rnode] {1,4};
\node (v5) at (240:3) [rnode] {1,5};
\node (v6) at (24:3)  [rnode] {1,6};
\node (w3) at (48:3) [wnode] {2,3};
\node (w4) at (0:3) [wnode] {2,4};
\node (w5) at (192:3) [wnode] {2,5};
\node (w6) at (312:3) [wnode] {2,6};
\node (u4) at (216:3) [wnode] {3,4};
\node (u5) at (288:3) [wnode] {3,5};
\node (u6) at (96:3) [wnode] {3,6};
\node (x5) at (72:3) [wnode] {4,5};
\node (x6) at (264:3) [wnode] {4,6};
\node (y6) at (144:3) [wnode] {5,6};
%%edges from v2
\draw (v2) -- (u4);
\draw (v2) -- (u5);
\draw (v2) -- (u6);
\draw (v2) -- (x5);
\draw(v2) -- (x6);
\draw(v2) -- (y6);
%%edges from v3
\draw (v3) -- (w4);
\draw (v3) -- (w5);
\draw (v3) -- (w6);
\draw (v3) -- (x5);
\draw (v3) -- (x6);
\draw (v3) -- (y6);
%%edges from v4
\draw (v4) -- (w3);
\draw (v4) -- (u5);
\draw (v4) -- (u6);
\draw (v4) -- (w5);
\draw (v4) -- (w6);
\draw (v4) -- (y6);
%%edges from v5
\draw (v5) -- (w3);
\draw (v5) -- (u4);
\draw (v5) -- (u6);
\draw (v5) -- (w4);
\draw (v5) -- (w6);
\draw (v5) -- (x6);
%%edges from v6
\draw (v6) -- (w3);
\draw (v6) -- (w5);
\draw (v6) -- (x5);
\draw (v6) -- (w4);
\draw  (v6) -- (u4);
\draw  (v6) -- (u5);
%%additional edges from w3
\draw  (w3) -- (x5);
\draw (w3) -- (x6);
\draw (w3) -- (y6);
%%additional edges from w4
\draw (w4) -- (u5);
\draw (w4) -- (u6);
\draw  (w4) -- (y6);
%%additional edges from w5
\draw  (w5) -- (u6);
\draw (w5) -- (x6);
\draw  (w5) -- (u4);
%%additional edges from w6
\draw (w6) -- (x5);
\draw  (w6) -- (u5);
\draw (w6) -- (u4);
%%additional edges from u4
\draw (y6) -- (u4);
%%additional edges from u5
\draw  (u5) -- (x6);
%%additional edges from u6
\draw  (u6) -- (x5);
        \end{tikzpicture}
    \caption{$KG(6,2)$.  The red vertices form an extremal design, and are a $t$-design for $t=1$ on the $(6,2)$ Johnson scheme.}
    \label{fig: KG(6,2)}
        \end{center}
\end{figure}
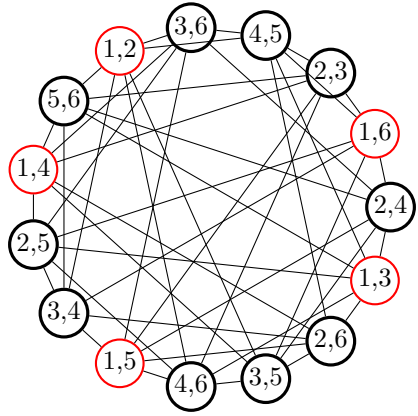

The case $s=2$ is omitted due to a technical detail which we briefly outline here. The $(l,2)$ Johnson scheme is cometric for both possible orderings of the eigenspaces: $J_0, J_1, J_2$ and $J_0, J_2, J_1$. The design $Y$ of Proposition \ref{johnson graph counterex} integrates one nontrivial eigenspace, which will come first in the $t$-design ordering for one of these cometric structures on the $(l,2)$ Johnson scheme. $Y$ is thus a $t$-design for $t=1$ in one cometric ordering. We now exhibit an optimal design in the graph $G_2$ of the $(8,3)$ Johnson scheme that is not a $t$-design.

\begin{example} \label{optimal but not t johnson}
 Let $(X,\cR)$ be the $(8,3)$ Johnson scheme, and let $Y\subset X$ be the set consisting of all vertices which contain the fixed element $\{1\}$.  Then $Y$ is a stable set in $G_3 = KG(8,3)$ which attains the Hoffman bound, as was shown in \cite{Golubev}. 
\begin{table}[h!]
    \centering
    \begin{tabular}{c|c |c c c }
    Eigenspace & $\dim(\L)$ &$G_1$ & $G_2$ & $G_3$ \\ \hline
    $\col(J_0)$ & 1 & 0 & 0 & 0 \\
    $\col(J_1)$ & 7 & $-8/15$ & $-16/15$ & $-8/5$ \\
    $\col(J_2)$ & 20 & $-14/15$ & $-7/6$ & $-7/10$ \\
    $\col(J_3)$ & 28 & $-6/5$ & $-9/10$ &  $-11/10$ 
    \end{tabular}
    \caption{The eigenvalues of $L = AD^{-1} -I$ by eigenspace for the graphs $G_1, G_2,$ and $G_3$ of the $(8,3)$ Johnson scheme.}
    \label{tab:J(8,3)spectra}
\end{table}
By Theorem \ref{myHoff}, $Y$ fails to integrate the eigenspace $\col(J_1)$, which has eigenvalue $ -16/15$ in $G_2$. The eigenspace $\col(J_1)$ is last in the frequency ordering for $G_2$. Thus $Y$ has efficacy $7/49$ on $G_2$. A brute force search shows that no subset of 6 or fewer vertices integrates $\col(J_2)$, the first nontrivial eigenspace by frequency for $G_2$. Thus $Y$ is an optimal design on $G_2$. We note that $Y$ is not a stable set in $G_2$. 
\end{example}

This indicates that the eigenspace ordering enforced by the concept of a classical $t$-design is unrelated to the frequency ordering. 

\subsection{The Hamming Scheme}

Take $X = \{0,1\}^d$ and say that $(x,y) \in R_i$ if $d_H(x,y) =i$.  Then $(X,\cR)$ forms an association scheme with $d$ classes, called the \emph{Hamming scheme} (see Figure \ref{fig: hamm scheme}).  The Hamming scheme is cometric for the ordering $\col( J_i) = \spanset\{ \phi_v : wt(v)=i\}$. Note that $G_1=Q_d$ and $G_{[N]}= Q_d(N)$. We further distinguish graphical designs from $t$-designs by comparing how the Hamming code performs in each setting. 

Like with Johnson schemes, classical $t$-designs in the Hamming scheme have been characterized in terms of a separate combinatorial object called orthogonal arrays.

\begin{definition}[see Section 11.8 of \cite{ECC}]
 \emph{Let $F$ be a set with $q$ elements. An  \emph{orthogonal array} is a $K \times d$ matrix with entries from $F$ such that any set of $i$ columns contains all possible $q^i$ row vectors exactly $\g $ times.  The array has size $K$, $d$ constraints, $q$ levels, strength $i$, and index $\g$, and is denoted $(K,d,q,i)$. An orthogonal array is \emph{linear} if $q$ is a prime power and the rows of the orthogonal array form a vector subspace of $F^d$. }
\end{definition}

\begin{example}  \emph{Consider the matrix 
\[ A = \begin{bmatrix}
1 & 1 & 1 \\
0 & 1 & 0 \\
1 & 0 & 0 \\
0 & 0 & 1 
\end{bmatrix}. \]
Here, $F= \{0,1\}$, so $q=2.$ Every choice of two columns of $A$ contains each vector in $\{0,1\}^2$ exactly once.  Thus $A$ is a $(4,3,2,2)$ orthogonal array of index 1. There is no zero row, so the rows of $A$ are not a subspace of $\{0,1\}^3$. Hence $A$ is not linear.}
\end{example}
In the Hamming scheme, $t$-designs are equivalent to orthogonal arrays.

\begin{theorem}[Theorem 4.4 of \cite{DelsarteThesis}] \label{hamm tdesigns are orthog arrays}
 \emph{$Y\subset X= \{0,1\}^d$ is a $t$-design in the Hamming scheme if and only if the vectors of $Y$ are a $(|Y|, d, 2,t)$ orthogonal array. }
\end{theorem}
We can use this result to compare how the Hamming code $H_r$ performs as a classical $t$-design versus as a graphical design. We also use the following result for codes and orthogonal arrays. 
\begin{proposition}[see Theorem 10.17 of \cite{Stinson}]
 \emph{Let $C \subset \{0,1\}^d$ be a linear code of dimension $K$. Then, $\dist(C) =\rho$ if and only if $C^\perp$ is a linear $(2^K, d, 2, \rho-1)$ orthogonal array. } \label{orthog array and dual dist}
\end{proposition}
The  Hamming code integrates nearly twice as many eigenspaces in the graphical design ordering than in the ordering from classical $t$-designs. 
\begin{proposition} \label{hamm better GD than TD}
 \emph{The Hamming code $H_r$ is a 
 $(2^r-1)$-graphical design on $Q_{2^r-1}$, but only a 
classical $t$-design for $t = 2^{r-1}-1$.} 
\end{proposition}

\begin{proof}
See Theorem \ref{Main Hamm} for the result that $H_r$ is a  $(2^r-1)$-graphical design on $Q_{2^r-1}$. Recall that $H_r^\perp$ is the simplex code, and $\dist(H_r^\perp) = 2^{r-1}$.  Thus $H_r$ forms a $(2^{2^r-r-1}, 2^r-1, 2, 2^{r-1}-1)$ orthogonal array by Proposition \ref{orthog array and dual dist}.  Hence $H_r$ is a $t$-design where $t = 2^{r-1}-1$ by Theorem  \ref{hamm tdesigns are orthog arrays}.  \end{proof}

\end{section}

\section*{Acknowledgements} The author would like to thank Rekha Thomas and Stefan Steinerberger for their guidance and help preparing this manuscript, Shahar Kovalsky for sharing code that assisted in computations, and Ferdinand Ihringer for noting some missing details in Section 6. Figures 1 and 3A-C are due to Stefan Steinerberger.
Figure 3D is used with permission, copyright ©2015, PRISM Climate Group, Oregon State University, http://prism.oregonstate.edu/normals/; retrieved 3 Dec 2020.
Computations were done in \cite{MATLAB:2020}.

\printbibliography

@book{HFT,
    title = {Harmonic Function Theory},
    author = {S. Axler and P. Bordon and W. Ramey},
    year = {1992},
    publisher = {Springer-Verlag New York}
}

@online{spherical,
    author = "M. J. Mohlenkamp",
    title = "A User's Guide to Spherical Harmonics",
    url  = "http://www.ohiouniversityfaculty.com/mohlenka/research/uguide.pdf",
    addendum = "(version: 10.18.2016)"
}

@incollection{Hoffman,
    author = "A. J. Hoffman",
    title = {On eigenvalues and colorings of graphs},
    editor = {B. Harris},
    booktitle   = "Graph Theory and Its Applications",
    year = 1970,
}

@article{SSLinderman,
issn = {0025-5718},
journal = {Mathematics of computation.},
pages = {1933--1952},
volume = {89},
publisher = {National Academy of Sciences-National Research Council,},
number = {324},
year = {2020},
title = {Numerical integration on graphs: Where to sample and how to weigh},
address = {Washington, D.C. :},
author = {G. C. Linderman and S. Steinerberger},
}

@article{riemanndesigns,
    author = "S. Steinerberger",
    title = {Spectral Limitations of Quadrature Rules and Generalized Spherical Designs},
    url = "https://arxiv.org/abs/1708.08736",
    journal = "IMRN",
    year = "2019",
    addendum ={(accepted)}
}

@article{graphdesigns,
issn = {0364-9024},
journal = {Journal of graph theory.},
pages = {253--267},
volume = {93},
publisher = {John Wiley & Sons,},
number = {2},
year = {2020},
title = {Generalized designs on graphs: Sampling, spectra, symmetries},
address = {New York,},
author = {S. Steinerberger},
}

@article{Golubev,
    author = "K. Golubev",
    title = {Graphical designs and extremal combinatorics},
    journal = {Linear Algebra and its Applications},
    volume = 604,
    page = {490-506},
    year = 2020
}

@book{ECC,
author = "F. J. MacWilliams and N. J. A. Sloane",
title = "The Theory of Error Correcting Codes",
year = {1977},
publisher = "North Holland",
}

@article{DelsarteThesis,
    author = "Ph. Delsarte",
    title = {An Algebraic Approach to the Association Schemes of Coding Theory},
    journal = {Philips Res. Repts Suppl.},
    volume = 10,
    year = 1973
}

@book{Stinson,
author = "D. R. Stinson",
title = "Combinatorial Designs: Construction and Analysis",
year = {2004},
publisher = "Springer-Verlag: New York",
}

@article{Tanner,
issn = {0196-5212},
journal = {SIAM journal on algebraic and discrete methods.},
pages = {287--293},
volume = {5},
publisher = {Society for Industrial and Applied Mathematics},
number = {3},
year = {1984},
title = {Explicit concentrators from generalized n-gons},
address = {Philadelphia,},
author = {R. M. Tanner},
}

@article{AM,
issn = {0095-8956},
journal = {Journal of combinatorial theory.},
pages = {73--88},
volume = {38},
publisher = {Academic Press},
number = {1},
year = {1985},
title = {$\l_1$, Isoperimetric inequalities for graphs, and superconcentrators},
address = {Orlando, Fla. :},
author = {N. Alon and V. D. Milman.},
}

@book{Sagan,
issn = {0072-5285},
volume = {203},
publisher = {Springer,},
isbn = {9780387950679},
year = {2001},
title = {The Symmetric Group: Representations, Combinatorial Algorithms, and Symmetric Functions},
edition = {2nd ed.},
address = {New York :},
author = { B. E. Sagan },
}

@book{BrouwerHaemers,
issn = {0172-5939},
publisher = {Springer},
isbn = {978-1-4614-1938-9},
year = {2012},
title = {Spectra of Graphs},
address = {New York, NY :},
author = {A. E. Brouwer and W. H. Haemers},
}

@article{VallentinSDPECC,
  author    = {F. Vallentin},
  title     = {Semidefinite programming bounds for error-correcting codes},
  journal   = {CoRR},
  volume    = {abs/1902.01253},
  year      = {2019},
  url       = {http://arxiv.org/abs/1902.01253},
  archivePrefix = {arXiv},
  eprint    = {1902.01253},
  bibsource = {dblp computer science bibliography, https://dblp.org}
}

@article{DGSspherical,
issn = {0046-5755},
journal = {Geometriae dedicata},
pages = {363--388},
volume = {6},
publisher = {Kluwer Academic Publishers},
number = {3},
year = {1977},
title = {Spherical codes and designs},
address = {[Dordrecht] :},
author = {Ph. Delsarte and J. M. Goethal and J. J. Seidel},
}

@article{sobolev,
issn = {0002-3264},
journal = {Dokl. Akad. Nauk SSSR.},
pages = {310--313},
volume = {146},
publisher = {Izd-vo Akademii nauk SSSR,},
year = {1962},
title = {Cubature formulas on the sphere which are invariant under transformations of finite rotation groups},
address = {Leningrad :},
author = {S. Solobev}
}

@article{Lebedev,
issn = {0044-4669},
journal = {Zh. Vchisl. Mat. Mat. Fiz. i.},
pages = {293--306},
volume = {16},
publisher = {Nauka},
number = {2},
year = {1976},
title = {Quadratures on the sphere},
address = {Moskva :},
author = {V. I. Lebedev},
}

@article{Weyl,
issn = {0025-5831},
journal = {Mathematische Annalen.},
pages = {441--479},
volume = {71},
publisher = {Druck und Verlag von B.G. Teubner},
number = {4},
year = {1912},
title = {Das asymptotische Verteilungsgesetz der Eigenwerte linearer partieller Differentialgleichungen (mit einer Anwendung auf die Theorie der Hohlraumstrahlung)},
address = {Leipzig :},
author = {H. Weyl},
}

@article{samplingAGO,
issn = {1053-587X},
journal = {IEEE transactions on signal processing a publication of the IEEE Signal Processing Society.},
pages = {3775--3789},
volume = {64},
publisher = {Institute of Electrical and Electronics Engineers,},
number = {14},
year = {2016},
title = {Efficient sampling set selection for bandlimited graph signals using graph spectral proxies},
address = {New York, NY :},
author = {A. Anis and A. Gadde and A. Ortega},
}

@article{samplingTEOC,
   title={Sampling signals on graphs: from theory to applications},
   volume={37},
   ISSN={1558-0792},
   url={http://dx.doi.org/10.1109/MSP.2020.3016908},
   DOI={10.1109/msp.2020.3016908},
   number={6},
   journal={IEEE Signal Processing Magazine},
   publisher={Institute of Electrical and Electronics Engineers (IEEE)},
   author={Y. Tanaka and Y.C. Eldar and A. Ortega, Antonio and G. Cheung},
   year={2020},
   month={11},
   pages={14–30}
}

@article{samplingTBD,
issn = {1053-587X},
journal = {IEEE transactions on signal processing a publication of the IEEE Signal Processing Society.},
pages = {4845--4860},
volume = {64},
publisher = {Institute of Electrical and Electronics Engineers,},
number = {18},
year = {2016},
title = {Signals on graphs: uncertainty principle and sampling},
address = {New York, NY :},
author = {M. Tsitsvero and S. Barbarossa and P. Di Lorenzo},
}

@book{ChungSpectral,
publisher = {Published for the Conference Board of the mathematical sciences by the American Mathematical Society,},
isbn = {9780821803158},
year = {1997},
title = {Spectral graph theory},
address = {Providence, R.I. :},
author = {F.R.K. Chung},
}

@article{HammingOG,
issn = {0005-8580},
journal = {The Bell System technical journal.},
pages = {147--160},
volume = {29},
publisher = {American Telephone and Telegraph Co},
number = {2},
year = {1950},
title = {Error detecting and error correcting codes},
address = {[Short Hills, N.J., etc.]},
author = {R.W.Hamming},
}

@misc{GGmanifolds,
      title={Optimal asymptotic bounds for designs on manifolds}, 
      author={B. Gariboldi and G. Gigante},
      year={2018},
      eprint={1811.12676},
      archivePrefix={arXiv},
      primaryClass={math.AP}
}

@article{BRVmanifoldsII,
issn = {0176-4276},
journal = {Constructive approximation.},
pages = {93--112},
volume = {41},
publisher = {Springer-Verlag New York,},
number = {1},
year = {2015},
title = {Well-separated spherical designs},
address = {New York, NY :},
author = {Bondarenko, A. and Radchenko, D. and Viazovska, M.},
}

@article{BRVmanifoldsI,
issn = {0003-486X},
journal = {Annals of mathematics.},
pages = {443--452},
volume = {178},
publisher = {Princeton University Press, etc},
number = {2},
year = {2013},
title = {Optimal asymptotic bounds for spherical designs},
address = {[Princeton, N.J., etc.],},
author = {Bondarenko, A. and Radchenko, D. and Viazovska, M.},
}

@article{BMalgebras,
author = "Bose, R. C. and Mesner, D. M.",
doi = "10.1214/aoms/1177706356",
fjournal = "Annals of Mathematical Statistics",
journal = "Ann. Math. Statist.",
month = "03",
number = "1",
pages = "21--38",
publisher = "The Institute of Mathematical Statistics",
title = "On linear associative algebras corresponding to association schemes of partially balanced designs",
url = "https://doi.org/10.1214/aoms/1177706356",
volume = "30",
year = "1959"
}

@book{MATLAB:2020,
year = {2020},
author = {MATLAB},
title = {version 9.9.0 (R2020b)},
publisher = {The MathWorks Inc.},
address = {Natick, Massachusetts}
}

@article{PesensonI,
issn = {0002-9947},
journal = {Transactions of the American Mathematical Society.},
pages = {5603--5627},
volume = {360},
publisher = {American Mathematical Society},
number = {10},
year = {2008},
title = {Sampling in Paley-Wiener spaces on combinatorial graphs},
address = {[Providence, R.I.] :},
author = {Pesenson, I. Z.},
}

@misc{PesensonII,
      title={Sampling by averages and average splines on Dirichlet spaces and on combinatorial graphs}, 
      author={I. Z. Pesenson},
      year={2019},
      eprint={1901.08726},
      archivePrefix={arXiv},
}

@article{PesensonIII,
issn = {0002-9947},
journal = {Transactions of the American Mathematical Society.},
pages = {4257--4269},
volume = {352},
publisher = {American Mathematical Society},
number = {9},
year = {2000},
title = {A sampling theorem on homogeneous manifolds},
address = {[Providence, R.I.] :},
author = {Pesenson,I. Z.},
}

@book{BCN,
publisher = {Springer-Verlag,},
isbn = {9783540506195},
year = {1989},
title = {Distance-regular graphs /},
address = {Berlin ; New York :},
author = {Brouwer, A. E. and Cohen, A. M. and A. Neumaier},
}

@misc{haemer,
      title={Hoffman's ratio bound}, 
      author={W. H. Haemers},
      year={2021},
      eprint={2102.05529},
      archivePrefix={arXiv},
      primaryClass={math.CO}
}

@article{Suzuki,
journal = {Journal of Algebraic Combinatorics.},
pages = {181-196},
volume = {7},
publisher = {Kluwer Academic Publishers},
year = {1998},
title = {Association schemes with multiple Q-polynomial structures},
address = {Netherlands},
author = {H. Suzuki},
}
\end{document}